\documentclass[11pt,times,twoside]{article}

\usepackage{graphicx,latexsym,euscript,makeidx,color,bm}
\usepackage{amsmath,amsfonts,amssymb,amsthm,thmtools,mathtools,mathrsfs,enumerate}
\usepackage[colorlinks,linkcolor=blue,anchorcolor=green,citecolor=red]{hyperref}

\usepackage{geometry}
\def\5n{\negthinspace \negthinspace \negthinspace \negthinspace \negthinspace }
\def\4n{\negthinspace \negthinspace \negthinspace \negthinspace }
\def\3n{\negthinspace \negthinspace \negthinspace }
\def\2n{\negthinspace \negthinspace }
\def\1n{\negthinspace }

\def\dbC{\mathbb{C}}     
     
\def\dbE{\mathbb{E}}     
\def\dbF{\mathbb{F}} \def\sF{\mathscr{F}}    
         
\def\dbH{\mathbb{H}}

\def\dbN{\mathbb{N}}     
     
\def\dbP{\mathbb{P}}     
     
\def\dbR{\mathbb{R}}     
\def\dbS{\mathbb{S}}     
     
   \def\cU{{\cal U}}

\def\Om{\Omega}

\def\ss{\smallskip}                
\def\ms{\medskip}                
               
\def\ds{\displaystyle}           
\def\ra{\rightarrow}      
 
\def\no{\noindent}        \def\q{\quad}                      
\def\ns{\noalign{\ss}}    \def\qq{\qquad}                    
    \def\hb{\hbox}                     
                   
         \def\rf{\eqref}                    
  \def\deq{\triangleq}               
            \def\({\Big (}
\def\les{\leqslant}                  \def\){\Big )}
\def\ges{\geqslant}       \def\esssup{\mathop{\rm esssup}}   \def\[{\Big[}
           \def\]{\Big]}
                   \def\cd{\cdot}

\def\nn{\nonumber}        \def\ts{\times}

\def\a{\alpha}           \def\g{\gamma}   \def\O{\Omega}   
\def\b{\beta}            \def\d{\delta}        
               
\def\e{\varepsilon}             
    \def\t{\tau}     \def\f{\varphi}  \def\i{\infty}   

\theoremstyle{plain}
\newtheorem{thm}{\bf Theorem}[section]

\newtheorem{lem}[thm]{\bf Lemma}
\newtheorem{prop}[thm]{\bf Proposition}
\newtheorem{defn}[thm]{\bf Definition}

\theoremstyle{rmk}
\newtheorem{rmk}[thm]{\bf Remark}

\makeatletter

\@addtoreset{equation}{section}
\makeatother

\allowdisplaybreaks[4]

\begin{document}

\title{{\bf Dynamic programming principle for delayed stochastic recursive optimal control problem and HJB equation with non-Lipschitz generator}\thanks{This paper is supported by National Key R\&D Program of China (Grant Nos.2022YFA1006101, 2018YFA0703900), National Natural Science Foundation of China (Grant Nos.12101291, 12371445), Guangdong Basic and Applied Basic Research Foundation (grant No. 2022A1515012017), the Taishan Scholars Climbing Program of Shandong (grant No. TSPD20210302), and Science and Technology Commission of Shanghai Municipality (Grant No.22ZR1407600).}
}


\author{
Jiaqiang Wen\thanks{Department of Mathematics, Southern University of Science and Technology, Shenzhen 518055, China (Email: {\tt wenjq@sustech.edu.cn}).
}~,~~~
Zhen Wu\thanks{School of Mathematics, Shandong University, Jinan 250100, China (Email:{\tt wuzhen@sdu.edu.cn)}.
}~,~~~
Qi Zhang\thanks{Corresponding author. School of Mathematical Sciences, Fudan University, Shanghai 200433, China ({\tt qzh@fudan.edu.cn}). }
}

\date{}
\maketitle


\no\bf Abstract. \rm
In this paper, we study the delayed stochastic recursive optimal control problem with a non-Lipschitz generator, in which both the dynamics of the control system and the recursive cost functional depend on the past path segment of the state process in a general form.
First, the dynamic programming principle for this control problem is obtained. Then, by the generalized comparison theorem of backward stochastic differential equations and the stability of viscosity solutions, we establish the connection between the value function and the viscosity solution
of the associated Hamilton-Jacobi-Bellman equation. Finally, an application to the consumption-investment problem under the delayed continuous-time Epstein-Zin utility with a non-Lipschitz generator is presented.

\ms

\no\bf Key words: \rm
delayed stochastic recursive control problem, non-Lipschitz generator, dynamic programming principle, Hamilton-Jacobi-Bellman equation, stochastic differential utility.

\ms

\no\bf AMS subject classifications. \rm
90C39, 60H10, 93E20.

\section{Introduction}

The classical stochastic control theory appears with the birth of stochastic analysis and developed rapidly in recent decades due to its wide range of applications (see Yong--Zhou \cite{Yong-Zhou 1999}), and one of its important applications is the optimal consumption-investment problem under stochastic differential utility (SDU, for short), which was put forward by Duffie--Epstein \cite{Duffie-Epstein-92} in a conditional expectation form. Actually,  SDU  is equivalent to the nonlinear backward stochastic differential equations (BSDEs, for short) and its control problem is the stochastic recursive control problem introduced firstly by Peng \cite{Peng-92} and developed by  Wu--Yu \cite{Wu-Yu 2008}, Li--Peng \cite{Li-Peng 2009}, Pu--Zhang \cite{Pu-Zhang-18}, 
to name but a few.
Due to the importance of BSDEs in modern mathematical finance (see  El Karoui--Peng--Quenez \cite{Karoui-Peng-Quenez 1997}), both stochastic recursive control theory and the optimal consumption-investment problem of stochastic differential utilities achieved great progresses in the past few decades.

\ms

On the other hand, the research of many natural and social phenomena shows that the future development of many processes depends on not only their present state but also essentially on their historical information.
For example, the immune response of the vaccine of COVID-19 depends on the complete history of the vaccinator's exposure to microbial infection antigens; 
the quality of a harvest does not only depend on the current temperature of the environment but also on the whole pattern of the past temperatures;
the price of the option depends not only on its current value but also on the whole history of the underlying price, etc.
In the theory of stochastic processes, such phenomenons can be described by the stochastic differential delayed equations (SDDEs, for short). See Mohammed \cite{Mohammed 1984,Mohammed 1998}, Chang \cite{Chang 2008}, and the references therein.

\ms

Hence, an interesting problem is whether the stochastic differential utilities can be developed to stochastic control problems with delay. The answer is positive, although most problems remain practically intractable due to the complex infinite-dimensional state-space framework. Here we name but a few existing results. For example, Larssen \cite{Larssen 2002} obtained the dynamic programming principle for the systems of stochastic control with delay. Later, Larssen--Risebro \cite{Larssen-Risebro 2003} further studied the optimal consumption problems based on the stochastic delay
equations and presented the financial application of dynamic programming principle. Arriojas--Hu--Mohammed--Pap \cite{Arriojas-Hu-Mohammed--Pap 2007} got a delayed Black--Scholes formula.
Fuhrman--Masiero--Tessitore \cite{Marco-2010} studied a forward-backward system with delay, its application to optimal stochastic control and the regularity of associated Hamilton-Jacobi-Bellman (HJB, for short) equations. Chen--Wu \cite{Chen-Wu-12} was concerned with one kind of delayed stochastic recursive optimal control problem, proved the dynamic programming principle, and demonstrated that the value function is a  viscosity solution of the corresponding HJB equation. Shi--Xu--Zhang \cite{Shi-Xu-Zhang 2015} studied a general kind of stochastic differential delayed equations, obtained the stochastic maximal principle, and show the application to stochastic portfolio optimization model with delay and bounded memory. Tang--Zhang \cite{Tang-Zhang-15} obtained the dynamic
programming principle in a path-dependent case and studied the associated path-dependent Bellman equations. Li--Wang--Wu \cite{Li-Wang-Wu 2020} focused on a class of linear-quadratic optimal control for time-delay stochastic recursive utility.
It should be pointed out that all the above-mentioned works concerning delayed systems are based on the condition of Lipschitz generator. While in many SDU models the generator is non-Lipschitz with respect to the utility and consumption, see \rf{22.1.21.3} below for an example.

\ms

In this paper, inspired by the financial applications, we are interested in studying the delayed stochastic recursive control problem with non-Lipschitz generator. The dynamic programming principle for the delayed systems will be obtained, and the connection between the value function and the viscosity solution
of the associated Hamilton-Jacobi-Bellman equation will be shown. In order to present our work more clearly, we present the problem in details.
Consider the following controlled SDDE

\begin{equation}\label{state-equation}
\left\{\begin{split}
\ds dX(s)&=b(s,X_s,v(s))dt+\sigma(s,X_s,v(s))dW(t),\q s\in[t,T],\\
\ns\ds X(s)&=\g(s),\q s\in[t-\delta,t],
\end{split}\right.
\end{equation}
where $\g$ comes from the continuous function space $C([0,t];\dbR^n)$ regarded as the past value of the equation, $X_s$ is the segment of the path of $X(\cd)$ from $s-\delta$ up to time $s$ for a fixed $\delta>0$, and $v(\cd)$ is an admissible control process arisen from the admissible control set $\cU$ (see \rf{22.1.21.1} below). For the state equation \rf{state-equation}, the cost functional is associated with the solution of the following BSDE for $s\in[t,T]$,
\begin{equation}\label{BSDE}
 Y^{t,\g;v}(s)=\Phi(X_T^{t,\g;v})+\int_{s}^{T} g(r,X_r^{t,\g;v},Y^{t,\g;v}(r),Z^{t,\g;v}(r),v(r))dr
 -\int_{s}^{T} Z^{t,\g;v}(r)dW(r),
\end{equation}
and the cost functional is defined as below
\begin{equation}\label{cost-functional}
J(t,\g;v)\deq Y^{t,\g;v}(t),\q t\in[0,T].
\end{equation}
Since the recursive utility can be regarded as the conditional expectation form of the solution of BSDE \rf{BSDE}, the stochastic recursive control system in form contains the following SDU
\begin{equation}\label{22.1.21.2}
U^{t,\g;c}(t)=\dbE^{\mathscr{F}_{t}}\left[\Phi(X_T^{t,\g;c})
+\int_{t}^{T} g\left(r,X_r^{t,\g;c}, U^{t,\g;c}(r),c(r)\right) d r\right],
\end{equation}
where $c(\cd)$ is a consumption process serving as the control. Indeed, \rf{22.1.21.2} is a general form of the SDU in Duffie--Epstein \cite{Duffie-Epstein-92}, involving the delay.
The goal of the optimal control problem is to find an optimal control process $\tilde v\in\cU$ to maximize the cost functional \rf{cost-functional} for given $(t,\g)$.

\ms

In this paper, we are committed to relaxing the Lipschitz restriction to the generator, i.e., the generator
$g(\g,u,c)$ in \rf{22.1.21.2} is continuous but not necessarily Lipschitz with respect to both the utility variable $u$ and the consumption variable $c$. For the variable $u$, the generator could be of polynomial growth in our assumptions.
These settings will make a lot of difference in the deduction of the dynamic programming principle and bring troubles in the verification of conditions for the stability of viscosity solutions, which plays the key role in the connection between the value function and the viscosity solution of the associated HJB equation.
Let's present an significant example for the non-Lipschitz generator with respect to utility and consumption, i.e., the well-known continuous-time Epstein-Zin utility with a form
\begin{equation}\label{22.1.21.3}
  g(u,c)=\frac{\vartheta}{1-\frac{1}{\psi}}(1-r) u\bigg[\Big(\frac{c}{\left((1-r) u\right)^{\frac{1}{1-r}}}\Big)^{1-\frac{1}{\psi}}-1\bigg],
\end{equation}
where $\vartheta>0$ represents the rate of time preference, $0<\psi \neq 1$ is the elasticity of intertemporal substitution, and $0<r \neq 1$ is the coefficient of relative risk aversion.
Note that the generator $g(u,c)$ in \rf{22.1.21.3} is non-Lipschitz in $u$ and $c$ but monotonic with respect to $u$ in a wide range of parameters.

\ms

The main difficulties of our problem come from the delayed term and the non-Lipschitz generator.
For the former one, we need to overcome the trouble from the infinite-dimensional of the space $C([0,t];\dbR^n)$ in the verification that the optimal value function satisfies the HJB equation. We manage to do it with the help of the semigroup property of generator in the infinite-dimensional setting, by borrowing ideas from Fuhrman--Masiero--Tessitore \cite{Marco-2010} and Chen--Wu \cite{Chen-Wu-12}.
The non-Lipschitz generator brings troubles in the establishment of the relationship between the value function and its associated HJB equation. We conquer it by applying the techniques for the solvability of BSDE with monotonic coefficients such as in Pardoux \cite{Pardoux 1999} and the stability of viscosity solutions in Fleming--Soner \cite{Fleming-Soner-06}.
%

\ms

The paper is organized as follows. In Section \ref{Sec2}, some preliminaries are presented. In Section \ref{Mainresults}, we give the main results, i.e.,
the dynamic programming principle for delayed stochastic recursive optimal control problem and the relationship between the value function of the delayed stochastic recursive optimal control problem and
the viscosity solution of associated HJB equation, which  develop the results as in \cite{Marco-2010,Chen-Wu-12,Pu-Zhang-18}. In Section \ref{Sec5}, an example concerning continuous-time Epstein-Zin utility with delay and non-Lipschitz generator is presented to demonstrate the applications of our theoretical results to mathematical finance. Some auxiliaries are presented in Section \ref{Auxiliary}.

\section{Preliminaries} \label{Sec2}

Let $(\Omega,\sF,\dbF,\dbP)$ be a complete filtered probability space and a $d$-dimensional standard Brownian motion $\{W(t)\;;0\les t<\i\}$ is defined in it, where $\dbF=\{\sF_t; 0\les t<\i\}$ is the natural filtration of the Brownian motion $W(\cd)$ with $\sF_0$ containing $\mathcal{N}$, the class of all $\dbP$-null sets  of $\sF$.
Let $T>0$ be a fixed terminal time and $\delta\ges 0$ be a given finite time delay.
For each $t\in[0,T]$, we denote by $\dbC_t\deq C([t-\delta,t];\dbR^n)$ the set of bounded continuous $\mathbb{R}^{n}$-valued paths on $[t-\delta,t]$ and
 $$\dbC=\bigcup_{t\in[0,T]}\dbC_{t}.$$
For each $\g_t\in \dbC_t$, it denotes the path from $t-\delta$ up to time $t$, i.e.,
\begin{align*}
\g_t=\gamma_t(\theta)_{t-\delta\les \theta\les t},\q~ t\in[0,T],
\end{align*}
and its value at time $r$ is denoted by $\gamma_t(r)$. 
For each $t, s\in[0,T]$ and  $\g_t,\g_s\in \dbC$, we denote
\begin{align*}
\|\gamma_t\|_\dbC\deq& \sup_{t-\delta\les r\les t}|\gamma_t(r)|,\\
\|\g_t-\g_s\|_\dbC\deq &\sup_{(t-\delta)\wedge(s-\delta)\les r\les t\vee s}|\gamma_t((t-\delta)\vee r\wedge t)-\gamma_s((s-\delta)\vee r\wedge s)|.
\end{align*}
Then, it is obvious that $\dbC_t$ is a Banach space with respect to $\|\cdot \|$.
Next, for $t\in[0,T]$, $q\ges1$ and Euclidean space $\dbH$ (which could be $\dbR^n$, $\dbR^{n\ts m}$, $\dbS^n$, etc.), we introduce the following spaces of functions and processes:
\begin{align*}
\ds L^{2 q}(\O,\sF_t;\dbH)\deq &\Big\{\xi:\O\to\dbH\bigm|\xi\hb{ is
  $\sF_t$-measurable, }\dbE|\xi|^{2 q}<\i\Big\},\\
\ns\ds M_\dbF^{2 q}(t,T;\dbH)\deq &\bigg\{\f:[t,T]\times\O\to\dbH\bigm|\f(\cd)\hb{ is
    $\dbF$-adapted and}\ \dbE\int^T_t|\f(s)|^{2 q}ds<\i\bigg\}, \\
S^{2 q}_{\dbF}(t, T ;\dbH)\deq &\bigg\{\f:[t,T]\times\O\to\dbH\bigm|\f(\cd)\hb{
    is $\dbF$-adapted, continuous and}\ \dbE\sup_{s\in[t,T]}|\f(s)|^{2q}<\i\bigg\},\\
\ns\ds L^{2q}_{\dbF}(\Om;\dbC_t)\deq&~\bigg\{\g:\Om\to\dbC_t\bigm|\g(\cd)\hb{ is
$\dbF$-adapted and}\ \dbE\|\g_t\|_\dbC^{2q}<\i\bigg\}
%
%
\end{align*}
and
\begin{equation}\label{22.1.21.1}
\cU\deq\Big\{v\in M_{\dbF}^2(0,T;\dbR^m)\big|
\hb{ the process $v$ takes value in a compact set $U\in\dbR^m$}\Big\}.
\end{equation}
Here $\cU$ is called the admissible control set.

\ms

For an element of the space $L^{2}_{\dbF}(\Om;\dbC_t)$, we have the following proposition.

\begin{prop}[Lemma 3.4 in \cite{Chen-Wu-12}]\label{9.22.2}\sl
For any bounded $\Upsilon\in L^{2}_{\dbF}(\Om;\dbC_t)$, there is a sequence
$$\Upsilon^n=\sum^n_{i=1}\g^iI_{A_i}$$
converging to $\Upsilon$ in the space $L^{2}_{\dbF}(\Om;\dbC_t)$, where
$\g^i\in\dbC_t$ with $1\les i\les n$, $n\in\mathbb{N}$ and $\{A_i\}^n_{i=1}$ is a partition of $(\Om,\sF_t)$.
\end{prop}
For the coefficients of the state equation \rf{state-equation}, we assume the conditions as follows.
\begin{itemize}
  \item [(A1)] $b(t,\g,v):[0, T] \times \dbC \times U \rightarrow \mathbb{R}^{n}$ and $\sigma(t, \g, v):[0, T] \times \dbC \times U \rightarrow \mathbb{R}^{n \times d}$ are joint measurable and continuous with respect to $t$.
  %
  \item [(A2)] There is a constant $L \ges 0$ such that for any $t\in[0,T]$, $\g,\bar\g\in\dbC$ and $v,\bar v\in U$, we have
         $$|b(t,\g,v)-b(t,\bar\g,\bar v)|+|\sigma(t,\g,v)-\sigma(t,\bar\g,\bar v)|
         \les L\left(\|\g-\bar\g\|_{\dbC}+|v-\bar v|\right).$$
\end{itemize}
For the sake of the proof of dynamic programming principle, we would like to consider a general stochastic differential equation (SDE, for short) with a random initial variable. By using a standard argument, we have  the following \autoref{21.9.1.4}, which is a non-essential extension of Theorem 2.3 in Fuhrman--Masiero--Tessitore \cite{Marco-2010}.

\begin{prop}\label{21.9.1.4}\sl
Assume (A1)-(A2) and $q \ges 1$. Then for any $t \in[0, T]$, $v \in \mathcal{U}$, $\Upsilon\in L^{2}_{\dbF}(\Om;\dbC_t)$, the following SDE
\begin{equation}\label{state-equation-1}
\left\{\begin{split}
\ds & dX^{t, \Upsilon ; v}(s)=b(s,X^{t, \Upsilon ; v}_s,v(s))dt+\sigma(s,X^{t, \Upsilon ; v}_s,v(s))dW(t),\q s\in[t,T],\\
\ns\ds & X^{t, \Upsilon ; v}(s)=\Upsilon_t(s),\q s\in[t-\delta,t],
\end{split}\right.
\end{equation}
has a unique strong solution $X^{t, \Upsilon ; v}(\cd) \in S_{\dbF}^{2 q}(t, T ; \mathbb{R}^{n})$.
Moreover, there is a constant $C\ges0$ depending only on $K,\ L$ and $T$ such that for any
$s\in[t, T], v, v^{\prime} \in \mathcal{U}, \Upsilon, \Upsilon^{\prime} \in L^{2}_{\dbF}(\Om;\dbC_t)$, we have
$$
\dbE^{\sF_t}\Big[\|X^{t, \Upsilon ; v}_s\|^{2q}_{\dbC}\Big]
\les C\Big(1+\|\Upsilon_t\|_{\dbC}^{2 q}+\dbE^{\sF_t}\int_{t}^{T}|v(s)|^{2 q} d s\Big)
$$
and
$$
\dbE^{\sF_t}\Big[\|X^{t, \Upsilon ; v}_s-X^{t, \Upsilon' ; v'}_s\|^{2q}_{\dbC}\Big]
\les C\Big(\|\Upsilon_t-\Upsilon'_t\|_{\dbC}^{2 q}+\dbE^{\sF_t}\int_{t}^{T}|v(s)-v'(s)|^{2 q} d s\Big).
$$
\end{prop}
With $X^{t,\Upsilon;v}(\cd)$ being the solution of the state equation \rf{state-equation-1}, we consider the following type of BSDE, for $s\in[t,T]$,
\begin{equation}\label{BSDE-2}
 Y^{t,\Upsilon;v}(s)=\Phi(X_T^{t,\Upsilon;v})+\int_{s}^{T} g(r,X_r^{t,\Upsilon;v},Y^{t,\Upsilon;v}(r),Z^{t,\Upsilon;v}(r),v(r))dr
 -\int_{s}^{T} Z^{t,\Upsilon;v}(r)dW(r).
\end{equation}
For the generator $(\Phi, g)$ of BSDE \rf{BSDE-2}, we give the following assumptions.
\begin{itemize}
  \item [(A3)] The generator $\Phi(\g):\dbC \rightarrow \mathbb{R}$ and $g(t,\g,y,z,v):[0, T] \times \dbC \times \mathbb{R}\times \mathbb{R}^{d} \times U \rightarrow \mathbb{R}$ are joint measurable, and $g(t, \g, y, z, v)$ is continuous with respect to $(t,y,v)$.
  \item [(A4)] There is a positive constant $\tilde L$ such that for any $t\in[0, T],\g,\bar\g\in\dbC,y\in\mathbb{R},z,\bar z\in\mathbb{R}^{d}, v\in U$,
      $$
      \left|\Phi(\g)-\Phi\left(\bar \g\right)\right|
      +\left|g(t,\g, y, z, v)-g\left(t, \bar\g, y, \bar z, v\right)\right|
      \les \tilde L\left(\left\|\g-\bar\g\right\|_\dbC+\left|z-\bar z\right|\right).$$
  \item [(A5)] There is a constant $\mu\in\dbR$ such that for any $t\in[0, T],\g\in\dbC,y,\bar y\in\mathbb{R},z\in\mathbb{R}^{d}, v\in U$,
      %
      $$(y-\bar y)[g(t,\g, y, z, v)-g(t, \g,\bar y, z, v)]\les \mu|y-\bar y|^2.$$
  \item [(A6)] For a given $p\ges1$, there is a constant $M\ges0$ such that for any $t\in[0, T],\g\in\dbC,y\in\mathbb{R},z\in\mathbb{R}^{d}, v\in U$,
      $$\left|g(t,\g, y, z, v)-g\left(t,\g,0, z, v\right)\right|\les M(1+|y|^p).$$
\end{itemize}
Note that (A3) implies that $g(t,0,0,0,v)$ is bounded for all $t\in[0,T]$ and $v\in U$. Moreover, under the condition (A1)-(A6), by a similar method to Pardoux \cite{Pardoux 1999}, we can prove that BSDE
\rf{BSDE-2} admits a unique adapted solution. Here, we conclude the existence, uniqueness, and some useful estimates to the solution of BSDE \rf{BSDE-2}.

\begin{prop}\label{9.20.2}\sl
Assume (A1)-(A6) and $q \ges 1$. Then for any $t \in[0, T],$ $v \in \mathcal{U}$, $\Upsilon\in L^{2}_{\dbF}(\Om;\dbC_t)$, BSDE \rf{BSDE-2} admits a unique strong solution
$(Y^{t, \Upsilon ; v}(\cdot),Z^{t,\Upsilon;v}(\cdot))\in S_{\dbF}^{2 q}(t, T ; \mathbb{R})\ts M_{\dbF}^2(t,T;\dbR^d)$.
Moreover, there is a constant $C\ges0$ depending only on $K, L, \tilde{L}, \mu, M$ and $T$ such that for any $s\in[t, T], v\in \mathcal{U}, \Upsilon, \Upsilon^{\prime} \in L^{2}_{\dbF}(\Om;\dbC_t)$, we have
\begin{align*}
\ds &\dbE^{\sF_t}\Big\{ \sup_{s\in[t,T]}|Y^{t,\Upsilon;v}(s)|^{2q}
+\int_t^T|Y^{t,\Upsilon;v}(s)|^{2q-2}|Z^{t,\Upsilon;v}(s)|^{2}\Big\}\\
\ns\ds\les &C\Big\{1+\|\Upsilon_t\|^{2q}_{\dbC}
+\dbE^{\sF_t}\int_t^T|g(r,0,0,0,v(r))|^{2q}dr \Big\}
\end{align*}
and
$$|Y^{t,\Upsilon;v}(t)-Y^{t,\Upsilon';v}(t)|\les C\|\Upsilon_t-\Upsilon'_t\|_{\dbC}.$$
\end{prop}
In the theory of the dynamic programming principle, the comparison theorem of BSDE plays a key role, without the exception of our circumstances that
the generator of BSDE is continuous and monotonic rather than the Lipschitz continuous with respect to its items. The following comparison theorem is an immediate consequence from Fan--Jiang \cite{Fan-Jiang-12} which is applicable under our circumstances.
\begin{prop}\label{9.24.2} \sl
Assume that $(\Phi, g)$ and $\left(\Phi^{\prime}, g^{\prime}\right)$ are two pair of generators of BSDE \rf{BSDE-2} with respect to the solutions $(Y, Z)$ and $\left(Y^{\prime}, Z^{\prime}\right)$, respectively.
Assume that for all $t\in[0, T],\g\in\dbC,
v\in U$,
$\Phi(\g) \les \Phi^{\prime}(\g)$
and
$$g(t,\g,Y^{\prime}(t),Z^{\prime}(t),v)\les g'(t,\g,Y^{\prime}(t),Z^{\prime}(t),v)\ \ {\rm (or}\ g(t,\g,Y(t),Z(t),v)\les g'(t,\g,Y(t),Z(t),v){\rm)}\ \ \  \hb{a.s.}$$
 %
%
%
%
 Then $Y(t)\les Y'(t)$ a.s., for any $t \in[0, T]$.
\end{prop}

\section{Main results}\label{Mainresults}

In this section, we present the main results, i.e., the dynamic programming principle and the connection between the value function and the viscosity solution of the associated Hamilton-Jacobi-Bellman equation.

\subsection{Dynamic programming principle}\label{Sec3}

In this subsection, we  show the dynamic programming principle of the stochastic recursive control problem, where the generator $g$ with delay is not necessarily Lipschitz but monotonic and continuous. First, we introduce a family of backward semigroups which was put forward by Peng \cite{Peng-97}.

\ms

For any $t\in[0,T], s\in(t,T],\g\in\dbC, v\in \cU$, and $\xi\in L^{2}(\O,\sF_s;\dbR)$, we define
$$
G_{\theta,s}^{t, \g ; v}[\xi] \triangleq \widetilde{Y}^{t, \g ; v}(\theta), \quad \theta \in[t, s],
$$
where $(\widetilde {Y}^{t, \g ; v}(\cd), \widetilde {Z}^{t, \g ; v}(\cd)) \in S_{\dbF}^{2 }(t, s; \mathbb{R})\ts M_{\dbF}^2(t,s;\dbR^d)$ is the solution of the following BSDE over the time interval $[t,s]$:
$$
\widetilde {Y}^{t, \g ; v}(\theta)=\xi
+\int_{\theta}^{s} f\big(r, X_{r}^{t,\g;v},\widetilde {Y}^{t,\g;v}(r),\widetilde{Z}^{t,\g;v}(r), v(r)\big)dr
-\int_{\theta}^{s} \widetilde {Z}^{t, \g ; v}(r) d W(r),
$$
and $X_{r}^{t, \g ; v}$ is the solution of the state equation \rf{state-equation} representing the segment of the path of the process $X^{t, x ; v}(\cd)$ from $r-\delta$ up to time $r$.
By the uniqueness of BSDE in \autoref{9.20.2}, for $\xi=Y^{t, \g ; v}(s)$, where $Y^{t, \g ; v}(\cd)$ is the solution of BSDE \rf{BSDE}, we have
$$
G_{t, T}^{t, \g ; v}[\Phi(X_{T}^{t, \g ; v})]
=G_{t, s}^{t, \g ; v}[Y^{t, \g ; v}(s)].
$$
On the other hand, coming back to the control system \rf{state-equation}-\rf{cost-functional}, we have that the related value function, which represents the control problem to maximize the cost functional, is defined by
\begin{equation}\label{value-function}
u(t,\g)\deq\esssup_{v\in\mathcal{U}}J(t,\g;v), \quad(t,\g)\in[0,T]\times\dbC.
\end{equation}
Similar to the classical situation, the value function $u(t,\g)$ defined in \rf{value-function} is a deterministic function in the setting of non-Lipschitz generator.

\begin{lem}\label{9.20.1} \sl
Assume (A1)-(A6). Then the value function $u(t,\g)$ defined in \rf{value-function} is a deterministic function.
\end{lem}

For clear presentation, the proof of the above lemma together with the proofs of the following lemmas are presented in the appendix (see Section \ref{Auxiliary}).

\begin{lem}\label{9.23.2} \sl
For any $t\in[0,T]$ and $\g,\bar\g\in\dbC$,
\begin{equation}\label{9.21.2}
|u(t,\g)|\les C\big( 1+\|\g\|_{\dbC}\big)
\end{equation}
and
\begin{equation}\label{9.21.1}
|u(t,\g)-u(t,\bar\g)|\les \tilde{C}\|\g-\bar\g\|_\dbC,
\end{equation}
where $C\ges0$ is the constant in \autoref{9.20.2}.
\end{lem}

Next, we would like to study some properties for the cost functional $J$ and the value function $u$ when the initial state is a $\dbC$-valued random variable, which will be used in the proof of the dynamic programming principle later.

\begin{lem}\label{9.23.1} \sl
	For any $t\in[0,T]$, $v\in\cU$ and $\Upsilon\in L^{2}_{\dbF}(\Om;\dbC_t)$, we have
	$$J(t,\Upsilon;v)=Y^{t,\Upsilon;v}(t).$$
\end{lem}

\begin{lem}\label{9.23.6}\sl
For any $t\in[0,T]$, $v\in\cU$ and $\Upsilon\in L^{2}_{\dbF}(\Om;\dbC_t)$, we have
\begin{equation}\label{9.23.3}
u(t,\Upsilon)\ges Y^{t,\Upsilon;v}(t)\q \hb{a.s.},
\end{equation}
and for arbitrary $\e>0$, there is an admissible control $v\in\cU$ such that
\begin{equation}\label{9.23.4}
u(t,\Upsilon)\les Y^{t,\Upsilon;v}(t)+\e\q \hb{a.s.}
\end{equation}
\end{lem}

Based on the above preparations, we can prove the first main result, i.e., the dynamic programming principle.

\begin{thm}[{\bf Generalized dynamic programming principle}]\label{GDDP} \sl
Assume (A1)-(A6). Then the value function $u(t,\g)$ defined in \rf{value-function} for our optimal control problem system obeys dynamic programming principle, i.e., for each $0\les \t\les T-t$,
\begin{equation}\label{9.24.1}
u(t,\g)=\esssup_{v\in\mathcal{U}}G_{t,t+\t}^{t,\g;v}[u(t+\t,X_{t+\t}^{t,\g;v})].
\end{equation}
\end{thm}

\begin{proof}
Noticing the existence and uniqueness of BSDE \rf{BSDE}, 
we have
$$
\begin{aligned}
\ds u(t, \g) &=\esssup_{v \in \mathcal{U}} G_{t, T}^{t, \g ; v}
[\Phi(X_{T}^{t, \g ; v})] =\esssup_{v \in \mathcal{U}} G_{t, t+\t}^{t, \g ; v}
[Y^{t, \g ; v}(t+\t)]=\esssup_{v \in \mathcal{U}} G_{t, t+\t}^{t, \g ; v}
[Y^{t+\t, X_{t+\t}^{t, \g ; v} ; v}(t+\t)].
\end{aligned}
$$
Keep in mind that the generator of BSDE \rf{BSDE} is not Lipschitz with respect to $Y$, which leads to the failure of the classical comparison theorem. Fortunately, a generalized comparison theorem in \autoref{9.24.2} given by
Fan--Jiang \cite{Fan-Jiang-12} works, in which the generator of BSDE is non-Lipschitz and ``weakly" monotonic with respect to the first solution in the solution pair. However, before applying the generalized comparison theorem, we need verify the necessary terminal condition, i.e., for $s\in[t,T]$ and $v\in\cU$, $u(s,X^{t,\g,v}_s)$ is square-integrable and can act as the terminal value of BSDE. For this, note that for arbitrary $\e>0$, \autoref{9.23.6} implies that there is an admissible control $\bar v\in\cU$ such that
$$
Y^{s, X_{s}^{t, \g ; v} ; \bar{v}}(s) \les
u(s, X_{s}^{t, \g ; v}) \les
Y^{s, X_{s}^{t, \g ; v};\bar{v}}(s)+\e.
$$
Hence we only need to show $\dbE|Y^{s, X_{s}^{t, \g ; v} ; \bar{v}}(s)|^2<\i$. Actually, it follows from (A1)-(A6), \autoref{21.9.1.4} and \autoref{9.20.2}, since
$$
\begin{aligned}
\ds \dbE|Y^{s, X_{s}^{t, \g ; v} ; \bar{v}}(s)|^2
\les C_{p}\left(1+\mathbb{E}\|X_{s}^{t, \g ; v}\|^{2}_{\dbC}\right) \les C_{p}\bigg(1+\|\g\|_{\dbC}^{2}+\dbE\int_{t}^{T}\left|v(r)\right|^{2} d r\bigg)<\infty.
\end{aligned}
$$
%
Now, by \autoref{9.24.2} we have
\begin{equation}\label{9.24.3}
u(t, \g) \les \esssup_{v \in \mathcal{U}}
G_{t, t+\t}^{t, \g ; v}[u(t+\t, X_{t+\t}^{t, \g ; v})].
\end{equation}
On the other hand, \autoref{9.23.6} also implies that, for arbitrary $\varepsilon>0$, there is an admissible control $\tilde{v} \in \mathcal{U}$ such that
$$
u(t+\t, X_{t+\t}^{t, \g ; v})-\e \les Y^{t+\t, X_{t+\t}^{t, \g ; v} ; \tilde{v}}(t+\t).
$$
Hence by \autoref{9.24.2} again, we have
\begin{align}\label{9.24.4}
\ds u(t, \g) & \ges \esssup_{v \in \mathcal{U}} G_{t, t+\t}^{t, \g ; v}[Y^{t+\t, X_{t+\t}^{t, \g ; v} ; \tilde{v}}(t+\t)] \nn\\
\ns\ds & \ges \esssup_{v \in \mathcal{U}} G_{t, t+\t}^{t, \g ; v}[u(t+\t, X_{t+\t}^{t, \g ; v})-\e] \nn\\
\ns\ds & \ges \esssup_{v\in\mathcal{U}}G_{t,t+\t}^{t,\g;v}[u(t+\t,X_{t+\t}^{t, \g ; v})]-\sqrt{C_{p}}\e,
\end{align}
where the last inequality is based on the regular estimation of BSDE (see  Peng \cite{Peng-92} or Zhang \cite{Zhang-17}).

\ms

Finally, due to the arbitrariness of $\e$, \rf{9.24.1} follows from \rf{9.24.3} and \rf{9.24.4}.
\end{proof}

\subsection{Hamilton--Jacobi--Bellman equation}\label{Sec4}

In the classical case, Peng \cite{Peng-92} established that the value function is as a viscosity solution to its associated Hamilton-Jacobi-Bellman  equation. In this subsection, our objective is to establish a connection between the value function defined in equation \rf{value-function} and the viscosity solution of the HJB equation \rf{HJB-equation} in the setting of a non-Lipschitz generator with delay.

\ms

To prove it, we need an additional assumption: namely, that the generator $g$ of BSDE \rf{BSDE} is independent of the second unknown variable $Z$, otherwise, obtaining the desired result would be significantly challenging. The good news is that this scenario finds important applications in finance (as described in Section \ref{Sec5}). Specifically, in this case, the generator $g(t, \gamma, y, z, v)$ in  BSDE \rf{BSDE} is simplified to $g(t, \gamma, y, v)$, and the HJB equation is a second-order fully nonlinear parabolic partial differential equation (PDE, for short) with the form
\begin{equation}\label{HJB-equation}
\left\{\begin{split}
\ds & \frac{\partial}{\partial t}u(t,\g)+H(t,\g,u(t,\g),\nabla_\g u(t,\g),\nabla_{\g}^2u(t,\g))=0,\q~
      (t,\g)\in[0,T)\ts\dbC,\\
\ns\ds & u(T,\g)=\Phi(\g),
\end{split}\right.
\end{equation}
where the Hamiltonian $H:[0,T]\ts\dbC\ts\dbR\ts\dbR^n\ts\dbS^n\rightarrow\dbR$ is defined as below
\begin{equation}\label{12.19.1}
H(t,\g,r,p,A)\deq\sup _{v\in U}\left\{\langle p, b(t,\g,v)\rangle
+\frac{1}{2}Tr\left(\sigma(t,\g,v)\sigma^*(t,\g,v)A\right)
+g\left(t,\g,r,v\right)\right\}.
\end{equation}
In the above, $\dbS^n$ denotes the matrix space of $n\ts n$ symmetric matrices, and the gradient $\nabla_\g u$ (resp., Hessian $\nabla_\g^2 u$) represents the first (resp., second) G\^ateaux derivative of $u$ with respect to $\g\in\dbC$, which is an $n$-tuple of finite Borel measures on $[0,T]$ and the dual space of $\dbC$. See e.g. Fuhrman--Masiero--Tessitore \cite{Marco-2010}.

\ms

Before recalling the definition of the viscosity solution of HJB equations, we specify the notation $C^{1,2}_{\text{Lip}}([0,T]\ts\dbC;\dbR)$ which is the space of all continuous functions from $[0,T]\ts\dbC$ to $\dbR$ such that their partial derivatives up to order one with respect to time variable and order two with respect to state variable are globally bounded and Lipschitz continuous.

\begin{defn} \sl
Let $u:[0,T]\ts\dbC\ra\dbR$ be a continuous function.
\begin{itemize}
  \item [{\rm (a)}] We call $u$ a viscosity subsolution of HJB equation \rf{HJB-equation}, if for any $\g\in\dbC$,  we have $u(T,\g)\les \Phi(\g)$, and for any
      {\rm $\phi\in C^{1,2}_{\text{Lip}}([0,T]\ts\dbC;\dbR)$}, $(t,\g)\in[0,T]\ts\dbC$, whenever $\phi-u$ attains a global minimum at $(t,\g)$, the following holds:
      $$\frac{\partial}{\partial t}\phi(t,\g)+H(t,\g,\phi,\nabla_\g \phi,\nabla_{\g}^2\phi)\ges0.$$
  \item [{\rm (b)}] We call $u$ a viscosity supersolution of HJB equation \rf{HJB-equation}, if for any $\g\in\dbC$,  we have $u(T,\g)\ges \Phi(\g)$, and for any
      {\rm $\phi\in C^{1,2}_{\text{Lip}}([0,T]\ts\dbC;\dbR)$}, $(t,\g)\in[0,T]\ts\dbC$, whenever $\phi-u$ attains a global maximum at $(t,\g)$, the following holds:
      $$\frac{\partial}{\partial t}\phi(t,\g)+H(t,\g,\phi,\nabla_\g \phi,\nabla_{\g}^2\phi)\les0.$$
  \item [{\rm (c)}] We call $u$ a viscosity solution of HJB equation \rf{HJB-equation}, if it is both a viscosity subsolution and viscosity supersolution.
\end{itemize}
\end{defn}

The definition of viscosity solution presented here is a delayed version of Crandall--Ishii--Lions \cite{Lions-92} which first introduces the theory of viscosity solutions for standard PDEs.
Besides, some preliminaries are needed. First of all, we point out the continuity of the values function defined in \rf{value-function} with respect to both the time and state variables.

\begin{prop}\label{9.14.1} \sl
Assume (A1)-(A6). The value function $u(t,\g):[0,T]\ts\dbC\ra\dbR$ defined in \rf{value-function} is continuous  with respect to $t$ and $\g$.
\end{prop}

\begin{rmk}\sl
A direct conclusion of \autoref{9.23.2} is that $u(t,\g)$ is Lipschitz continuous with respect to $\g$, and the continuity of $u(t,\g)$ with respect to $t$ is similar to the proof of Proposition 5.5 in Peng \cite{Peng-92}. There is no essential difference between the classical Lipschitz generator and the non-Lipschitz generator in our setting, so we leave out the proof of \autoref{9.14.1} here.
\end{rmk}

Based on the generator $g$ of BSDE \rf{BSDE} for $n\in\mathbb{N}$, we define a sequence of smooth functions $g_{n}$ as below:
\begin{equation}\label{gn}
g_{n}(t, \g, y, v) \triangleq\left(\rho_{n} * g(t, \g, \cdot, v)\right)(y),
\end{equation}
where $\rho_{n}: \mathbb{R} \rightarrow \mathbb{R}^{+}$ is a family of sufficiently smooth functions with the compact supports in $\left[-\frac{1}{n}, \frac{1}{n}\right]$ and satisfies
$$
\int_{\mathbb{R}} \rho_{n}(a) d a=1.
$$
As a result, concerning the generators $g_{n}$ for $n \in \mathbb{N}$, we can construct a sequence of BSDEs, for $s\in[t,T]$,
\begin{equation}\label{BSDE-gn}
 Y^{t,\g,n;v}(s)=\Phi(X_T^{t,\g;v})+\int_{s}^{T} g_n(r,X_r^{t,\g;v},Y^{t,\g,n;v}(r),v(r))dr
 -\int_{s}^{T} Z^{t,\g,n;v}(r)dW(r).
\end{equation}
Then, based on the above BSDE \rf{BSDE-gn}, the cost functional of a sequence of stochastic recursive control problems can be defined similarly, i.e., for $n\in\dbN$,
$$J_n(t,\g;v)\deq Y^{t,\g,n;v}(t),\q \hb{for $t\in[0,T],\ \g\in\dbC,\ v\in\cU$}.$$
At the same time, the value function of a sequence of stochastic recursive control problems has the corresponding forms,
for $n \in \mathbb{N}$,
\begin{equation}\label{10.13.1}
u_{n}(t, \g) \triangleq \esssup_{v \in \mathcal{U}} J_{n}(t, \g ; v), \q
\hb{for $t\in[0,T],\ \g\in\dbC$}.
\end{equation}
Also, Hamiltonian appears in the following form, for $n \in \mathbb{N}$,
\begin{equation}\label{10.24.1}
H_n(t,\g,r,p,A)\deq\sup _{v\in U}\left\{\langle p, b(t,\g,v)\rangle
+\frac{1}{2}Tr\left(\sigma(t,\g,v)\sigma^*(t,\g,v)A\right)
+g_n\left(t,\g,r,v\right)\right\}.
\end{equation}

Besides, we need a series of lemmas to prove the main result in this subsection.
The following \autoref{gn-g} proves that  the smoother generators $g_n$ defined in \rf{gn} are uniformly convergent to the generator $g$ in a compact subset of their domain.
\autoref{Yn-Y} implies the uniform convergence of the solutions $Y^{t, \g, n ; v}$ of BSDEs \rf{BSDE-gn} with smoother generators to the solution $Y^{t, \g ; v}$ of BSDE \rf{BSDE}  in the space of $L^{2 q}(\O,\sF_t;\dbR)$.
\autoref{un-u} displays that the value functions $u_n$ defined in \rf{10.13.1} are uniformly convergent to the value function $u$ defined in \rf{value-function} in a compact subset of their domain.
Based on the uniform convergence of $g_n$, $Y^{t, \g, n ; v}$, and $u_n$ to $g$, $Y^{t, \g; v}$, and $u$, respectively, it comes without a surprise in \autoref{hn-h} that the Hamiltonians $H_n$ defined in \rf{10.24.1} are convergent to $H$ defined in \rf{12.19.1} as well.
To end the preparation, we prove the stability property of viscosity solutions in \autoref{10.24.2} in order to obtain the connection between the value function $u$ defined in \rf{value-function} and the viscosity solution of HJB equation \rf{HJB-equation}.
Indeed, all the proofs of these mentioned lemmas will be presented in Section \ref{Auxiliary}.

\begin{lem}\label{gn-g} \sl
Assume (A1)-(A6). Then the sequence $g_{n}$ defined in \rf{gn} converges to $g$ uniformly in each compact subset of their domain.
\end{lem}


\begin{lem}\label{Yn-Y} \sl
Assume (A1)-(A6). Then for all $v\in\cU$, we have
$$
\lim _{n \rightarrow \infty} \sup _{(t, \g) \in K}
\dbE\left[\left|Y^{t, \g, n ; v}(t)-Y^{t, \g ; v}(t)\right|^{2}\right]=0,
$$
where $K$ is an arbitrary compact subset of $[0,T]\ts\dbC$, and $Y^{t, \g ; v}(\cd)$ and $Y^{t, \g, n ; v}(\cd)$ are the solutions of BSDE \rf{BSDE} and BSDE \rf{BSDE-gn}, respectively.
\end{lem}

\begin{lem}\label{un-u} \sl
Assume (A1)-(A6). Then the sequence $u_{n}$ defined in \rf{10.13.1} converges to $u$ defined in \rf{value-function} uniformly in each compact subset of their domain.
\end{lem}

\begin{lem}\label{hn-h} \sl
Assume (A1)-(A6). Then the sequence $H_{n}$ defined in \rf{10.24.1} converges to $H$ defined in \rf{12.19.1} uniformly in each compact subset of their domain.
\end{lem}

\begin{lem}[{\bf Stability}]\label{10.24.2} \sl
Let $u_{n}$ be a viscosity subsolution (resp. supersolution) to the following $\mathrm{PDE}$
$$
\frac{\partial}{\partial t} u_{n}(t, \g)+H_{n}\left(t, \g, u_{n}(t, \g), D_{x} u_{n}(t, \g), D_{x}^{2} u_{n}(t, \g)\right)=0, \quad(t, \g) \in[0, T) \times\dbC,
$$
where $H_{n}(t, \g, r, p, A):[0, T] \times \dbC \times \mathbb{R} \times \mathbb{R}^{n} \times \mathbb{S}^{n} \rightarrow \mathbb{R}$ is continuous and satisfies the ellipticity condition
\begin{equation}\label{12.10.1}
H_{n}(t, \g, r, p, X) \les H_{n}(t, \g, r, p, Y), \quad \text { whenever } X \les Y.
\end{equation}
Assume that $H_{n}$ and $u_{n}$ converge to $H$ and $u$, respectively, uniformly in each compact subset of their own domains. Then $u$ is a viscosity subsolution (resp. supersolution) of the limit equation
\begin{equation}\label{11.4.1}
\frac{\partial}{\partial t} u(t, \g)+H\left(t, \g, u(t, \g), D_{x} u(t, \g), D_{x}^{2} u(t, \g)\right)=0.
\end{equation}
\end{lem}

Based on above lemmas, we are going to establish the connection between the value function  defined in \rf{value-function}
and the viscosity solution of HJB equation \rf{HJB-equation}.

\begin{thm}\label{u-HJB} \sl
Assume (A1)-(A6) hold. Then $u$ defined in \rf{value-function} is a viscosity solution of HJB equation \rf{HJB-equation}.
\end{thm}

\begin{proof}
We divide the proof into two steps.

\ms

\textbf{Step 1}. Assume that for all $(t,\g,v) \in[0,T]\times\dbC\times U$,  $|g(t,\g,0,v)|$ is uniformly bounded.

\ms

Then from the assumption, we have the global Lipschitz condition of $g_{n}(t, \g, y, v)$ with respect to $y$. In fact, for any $(t, \g, v) \in[0, T] \times \dbC \times U$ and  $y,\bar y \in \mathbb{R}$, (A6) implies
$$
\begin{aligned}
\ds &\left|g_{n}\left(t, \g, y, v\right)-g_{n}\left(t, \g,\bar y, v\right)\right| \\
\ns=&\ds\left|\int_{|a| \les \frac{1}{n}} g(t,\g, a, v)\left[\rho_{n}\left(y-a\right)
-\rho_{n}\left(\bar y-a\right)\right] d a\right|\\
\ns\les &\ds \max _{a \in\left[-\frac{1}{n}, \frac{1}{n}\right]}|g(t, \g, a, v)| \int_{|a|
\les \frac{1}{n}}\left|\rho_{n}\left(y-a\right)-\rho_{n}\left(\bar y-a\right)\right| d a \\
\ns\les &\ds \max _{a \in\left[-\frac{1}{n}, \frac{1}{n}\right]}\left[|g(t, \g, 0, v)|
+M\left(1+|a|^{p}\right)\right] \int_{|a| \les \frac{1}{n}} C_{p}(n)\left|y-\bar y\right| d a \\
\ns\les &\ds C_{p}(M, n)\left|y-\bar y\right|.
\end{aligned}$$
So $g_n$ satisfies the assumption of Theorem 7.3 in Peng \cite{Peng-97}, by which we know that $u_{n}(t, \g)$ is a viscosity solution of the following equation, for $n\in\mathbb{N}$,
$$
\left\{\begin{array}{l}
\ds \frac{\partial}{\partial t} u_{n}(t, \g)+H_{n}\left(t, \g, u_{n}(t, \g), \nabla_{\g} u_{n}(t, \g), \nabla_{\g}^{2} u_{n}(t, \g)\right)=0, \quad~
(t, \g) \in[0, T) \times \dbC, \\
\ns\ds u_{n}(T, \g)=\Phi(\g).
\end{array}\right.
$$
Here $u_{n}$ and $H_{n}$ are defined in \rf{10.13.1} and \rf{10.24.1}, respectively.
Now, by \autoref{gn-g}-\ref{hn-h}, we have that the uniform convergence of $g_{n}$ to $g$, $Y^{t, \g, n ; v}(t)$ to $Y^{t, \g; v}(t)$, $u_{n}$ to $u$ and $H_{n}$ to $H$ holds in each compact subset of their own domains as $n \rightarrow \infty$. In addition, $H_{n}$ satisfies the ellipticity condition \rf{12.10.1}. Hence, from the stability of viscosity solution (\autoref{10.24.2}), we have that $u$ is a viscosity solution of the limit equation
$$
\frac{\partial}{\partial t} u(t, \g)+H\left(t, \g, u(t, \g), \nabla_{\g} u(t, \g), \nabla_{\g}^{2} u(t,\g)\right)=0.
$$
As for the terminal value of the above equation, by Lemma \ref{un-u}, we have
\begin{equation*}
u(T, \g)=\lim_{n\to\infty}u_n(T, \g)=\Phi(\g).
\end{equation*}
Thereby $u$ is a viscosity solution of HJB equation \rf{HJB-equation}.

\ms

\textbf{Step 2}. For all $(t, \g, v) \in$ $[0, T] \times \dbC \times U$, $|g(t, \g, 0, v)|$ is not necessary to be uniformly bounded.

\ms

In order to prove the result, we construct the sequence of functions
$$g_{m}(t, \g, y, v) \triangleq g(t, \g, y, v)-g(t, \g, 0, v)+\Pi_{m}(g(t, \g, 0, v)), \q \hb{for}\q m \in \mathbb{N},$$
where
$$\Pi_{m}(x)=\frac{\inf (m,|x|)}{|x|} x.$$
For these $g_{m}$, $m \in \mathbb{N}$, we get a sequence of BSDEs on the interval $[t, T]$
$$Y^{t, \g, m ; v}(t)=\Phi\left(X_{T}^{t, \g ; v}\right)
+\int_{t}^{T} g_{m}\left(s, X_{s}^{t, \g ; v}, Y^{t, \g, m ; v}(s), v(s)\right) d s
-\int_{t}^{T} Z^{t, \g, m ; v}(s) d W(s).$$
Similarly, we define the corresponding cost functional
$$
J_{m}(t, \g ; v) \triangleq Y^{t, \g, m ; v}(t), \quad \text { for }\ v \in \mathcal{U},\ t \in[0, T],\ \g \in \dbC,
$$
the value function
$$
u_{m}(t, \g) \triangleq \operatorname{esssup}_{v \in \mathcal{U}} J_{m}(t, \g ; v), \quad \text { for }\ t \in[0, T],\ \g \in \dbC,
$$
and the Hamiltonian, for $(t,\g,r,p,A)\in [0,T]\ts\dbC\ts\dbR\ts\dbR^n\ts\dbS^n$,
\begin{equation*}
H_m(t,\g,r,p,A)\deq\sup _{v\in U}\left\{\langle p, b(t,\g,v)\rangle
+\frac{1}{2}Tr\left(\sigma(t,\g,v)\sigma^*(t,\g,v)A\right)
+g_m\left(t,\g,r,v\right)\right\}.
\end{equation*}
Note that $g_{m}(t, \g, 0, v)=\Pi_{m}(g(t, \g, 0, v))$, which implies that $g_{m}(t, \g, 0, v)$ is uniformly bounded. In addition, it is easy to check that each $g_{m}$ satisfies the assumptions (A3)-(A6). Hence $g_{m}$ satisfies the assumptions of Step 1.
From Step 1, we see that $u_{m}$ is a viscosity solution of the following equation
\begin{equation}\label{10.25.1}
\left\{\begin{array}{l}
\ds \frac{\partial}{\partial t} u_{m}(t, \g)+H_{m}\left(t, \g, u_{m}(t, \g), D_{x} u_{m}(t, \g), D_{x}^{2} u_{m}(t, \g)\right)=0, \quad~
(t, \g) \in[0, T) \times \dbC, \\
\ns\ds u_{m}(T, \g)=\Phi(\g).
\end{array}\right.
\end{equation}
Next, we prove that as $m \rightarrow \infty$, the uniform convergence of $g_{m}$ to $g$, $Y^{t, \g, m ; v}(t)$ to $Y^{t, \g ; v}(t)$, $u_{m}$ to $u$ and $H_{m}$ to $H$ holds in each compact subset of their own domains too. It should be pointed out that due to the different definitions of $g_{m}$ here from $g_{n}$ in \rf{gn}, only the proof for the convergence of $g_{m}$ to $g$ is very different from \autoref{gn-g}, and the other convergence can be proved similarly as the proof of \autoref{Yn-Y}-\ref{hn-h} in turn.

\ms

From (A3), we have the continuity of $g$, so for each compact set $K \subset[0, T] \times \dbC\times \mathbb{R} \times U$ and any $(t, \g, y, v) \in K$, $g(t, \g, y, v)$ is bounded by a positive integer $M_{K}$. This implies the uniform convergence of $g_{m}$ to $g$ in $K$.
Hence, when $m \ges M_{K}$, we have
$$
\begin{aligned}
\sup _{(t, \g, y, v) \in K}\left|g_{m}(t, \g, y, v)-g(t, \g, y, v)\right|=\sup _{(t, \g, y, v) \in K}\left|g(t, \g, 0, v)-\Pi_{m}(g(t, \g, 0, v))\right|=0.
\end{aligned}
$$
Hence as $m \rightarrow \infty$,
$g_{m}$ uniform converges to $g$ in each compact subset of their domain.

\ms
Finally, by \autoref{10.24.2} again, even if $|g(t, \g, 0, v)|$ is not necessary to be uniformly bounded for any $(t, \g, v) \in[0, T] \times \dbC\times U$, we have that $u$ satisfies the limit equation of \rf{10.25.1}, which together with the fact $u(T, \g)=\Phi(\g)$ shows that $u$ is a viscosity solution of HJB equation \rf{HJB-equation}. The proof is complete.
\end{proof}

\section{Applications}\label{Sec5}

As stated in the Introduction, the stochastic differential formulation of recursive utility was introduced by Duffie and Epstein \cite{Duffie-Epstein-92}, which is actually a stochastic recursive control system. In this section, based on the connection between SDU and BSDE, we present an example to illustrate the application of our study. It should be pointed out that the inspiration of the model comes from Ivanov--Swishchuk \cite{Ivano-Swishchuk-05}, and for completeness, we would like to present the model in details.

\ms

Suppose that an investor can get a profit from his/her invested production. Denote the capital, the labor, and the consumption rate of the investor at time $t$ by $X(t)$, $\pi(t)$ and $c(t)$, respectively.
In 1928,  based on the assumption that the income of production is a general function of the capital and the labor, the following model was introduced by Ramsey \cite{Ramsey-1928}: 
$$
\frac{\mathrm{d} X(t)}{\mathrm{d} t}=f(X(t), \pi(t))-c(t).
$$
As we all know, when his/her investment contains a risky asset, such as stock, options, and futures, the investment involves uncertainty. Hence we would like to extend Ramsey's model to the following stochastic situation:
$$
\mathrm{d} X(t)=[f(X(t), \pi(t))-c(t)] \mathrm{d} t+\sigma(X(t)) \mathrm{d} W(t).
$$
%
%
However, the fitness of the model has been questioned for the basic assumption that there is no delay through the investment, which is not reasonable. In fact, the empirical evidence presents to us that the volatility of investment depends not only on the current state of the process but through its whole history.
%
%
Actually, as some examples shown in the Introduction, it is natural that the future development of many natural and social phenomena depends on their present state and historical information.
Therefore, it is reasonable to believe that in a consistent model of his investment, we should consider the past value of the capital during the investment. Hence, our model is further modified as
\begin{equation*}
\left\{\begin{aligned}
\ds \mathrm{d} X(t)=&\ [f(X_t, \pi(t))-c(t)] \mathrm{d} t+\sigma(X_t) \mathrm{d} W(t),\q t\in[0,T],\\
\ns\ds X(t)=&\ \g(t), \q t \in[-\delta, 0],
\end{aligned}\right.
\end{equation*}
where $X_t$ represents the segment of the path of the process $X(\cdot)$ from $t-\d$ to $t$.  A special case is that $X_t=X(t-\d)$ as in Chen--Wu \cite{Chen-Wu-12,Chen-Wu-10}. The function $f$ can be reasonably specific. For example, $f(X_t,\pi(t))=K f(X_t)^{\alpha} \pi^{\beta}(t)$, where $K, \alpha, \beta$ are some constants as in Gandolfo \cite{Gandolfo-96}. For simplicity, similar to Ivanov--Swishchuk \cite{Ivano-Swishchuk-05}, we set $\a=\b=1$ to rewrite our model as
%
%
%
\begin{equation}\label{wz4}
\left\{\begin{aligned}
\ds \mathrm{d} X(t)=&\ [K\pi(t) f(X_t)-c(t)] \mathrm{d} t+\sigma(X_t) \mathrm{d} W(t),\q t\in[0,T],\\
\ns\ds X(t)=&\ \gamma(t), \q t \in[-\delta, 0],
\end{aligned}\right.
\end{equation}
where $f(\g):\dbC\rightarrow \mathbb{R}$ and $\sigma(\g):\dbC\rightarrow \mathbb{R}$, and we require that the coefficients of (\ref{wz4}) satisfy (A1)-(A2) for $n=d=1$.
Now, we suppose that the SDU preference of the investor is a continuous time Epstein-Zin utility, and the utility is depicted by BSDE
\begin{equation}\label{12.7.2}
\left\{\begin{array}{l}
\ds d V(t)=-\frac{\vartheta}{1-\frac{1}{\psi}}(1-r) V(t)\left[\left(\frac{c(t)}{\left((1-r) V(t)\right)^{\frac{1}{1-r}}}\right)^{1-\frac{1}{\psi}}-1\right] d t+Z(t) d W(t), \\
\ns\ds V(T)=h\left(X_{T}\right),
\end{array}\right.
\end{equation}
where $h(\g): \dbC \rightarrow \mathbb{R}$ is a given Lipschitz continuous function. 
In BSDE \rf{12.7.2}, the coefficient $\vartheta>0$ represents the rate of time preference, the coefficient $r\in(0,1)\cup(1,\i)$ is the relative risk aversion, and $0<\psi \neq 1$ represents the elasticity of intertemporal substitution.
The investor's optimization goal is to maximize the utility
$
\max _{(\pi, c) \in \mathcal{U}} V(0)
$,
where
$$
\mathcal{U} \triangleq\left\{(\pi, c) \mid(\pi, c):[0, T] \times \Omega \rightarrow[a_1,a_2] \times\left[b_{1}, b_{2}\right]\ \text {and}\ \dbF \text {-adapted},\ \text {for given}\ a_1,a_2,b_1,b_2\in\mathbb{R}\right\}
$$
is the admissible control set.
Clearly, the generator of BSDE \rf{12.7.2} does not satisfy the Lipschitz condition with respect to the utility and the consumption at all times, however,  we can find the applications of our study in the non-Lipschitz situation.
Note that there are four cases that the generator is monotonic with respect to the utility as shown by Proposition 3.2 in Kraft--Seifried--Steffensen \cite{KSS-13}.
In view of the polynomial growth condition (A6) with respect to the utility, here we choose two situations for further investigation:
\begin{itemize}
  \item [(i)] $r>1$ and $\psi>1$,
  \item [(ii)] $r<1$ and $\psi<1$.
\end{itemize}

In both situations (i) and (ii), we can find suitable powers of utility such that the generator of BSDE \rf{12.7.2} is continuous and monotonic but non-Lipschitz with respect to the utility in its domain.
As for the continuity with respect to the consumption, it is obvious that both situations  are Lipschitz continuous if $b_{1}>0$. In particular, in the case of $b_{1}=0$, only the situation (i) satisfies the continuity but not Lipschitz continuity with respect to the consumption.

\ms

Therefore, under the assumptions (A3)-(A6), for all suitable non-Lipschitz situations, applying \autoref{u-HJB} we know that the value function of the investor is a viscosity solution of HJB equation
\begin{equation*}
\left\{\begin{split}
\ds &\sup _{(\pi,c)\in [a_1,a_2] \times[b_{1}, b_{2}]}
\Bigg\{\frac{\partial}{\partial t}u(t,\g)+\frac{1}{2}\sigma^2(\g)\nabla_\g^2 u(t,\g)+
 [K\pi f(\g)-c] \nabla_\g u(t,\g)\\
\ns\ds&\qq\qq\q + \frac{\vartheta}{1-\frac{1}{\psi}}(1-r) u(t,\g)\bigg[\bigg(\frac{c}{\left((1-r) u(t,\g)\right)^{\frac{1}{1-r}}}\bigg)^{1-\frac{1}{\psi}}-1\bigg]
\Bigg\}=0,\q (t,\g)\in[0,T)\ts\dbC,\\
\ns\ds & u(T,\g)=h(\g).
\end{split}\right.
\end{equation*}

\section{Proof of the auxiliary results}\label{Auxiliary}

We first give the detailed proofs to the lemmas in Subsection \ref{Sec3}.

\begin{proof}[Proof of \autoref{9.20.1}]
	The absence of a Lipschitz condition for the generator $g(t, \gamma, y, z, v)$ with respect to $y$ and $v$ brings about significant differences and challenges compared to the work of Peng \cite{Peng-92}.
	
	\ms
	
	To begin with, we would like to denote by $\left\{\sF_{t}^{s}; t \les s \les T\right\}$ the natural
	filtration of the Brownian motion $\{W(s)-W(t); t \les s \les T\}$ augmented by the $\dbP$-null sets of $\sF$, and define two subspaces of $\mathcal{U}$:
	$$\begin{aligned}
		&\mathcal{U}^t\deq\Big\{v \in \mathcal{U}\ \big| \
		v(s) \text { is } \mathscr{F}_{t}^{s} \text {-measurable for } t \les s \les T\Big\}, \\
		&\overline{\mathcal{U}}^t\deq\Big\{v \in \mathcal{U}\ \big| \
		v(s)=\sum_{j=1}^{N} v^{j}(s) 1_{A_{j}},\hb{ where } v^{j} \in \mathcal{U}^{t}\hb{ and } \{A_{j}\}_{j=1}^{N} \text { is a partition of }(\Omega, \sF_{t})\Big\}.
	\end{aligned}$$
	
	\textbf{Step 1}. We prove
	\begin{equation}\label{21.9.1.0}
		\esssup_{v \in \cU} J(t, \g ; v)=\esssup_{v \in \overline{\cU}^t} J(t, \g ; v).
	\end{equation}
	From the fact that $\overline{\cU}^t$ is a subset of $\mathcal{U}$, we have
	\begin{equation}\label{21.9.1.1}
		\esssup_{v \in \cU} J(t, \g ; v)\ges\esssup_{v \in \overline{\cU}^t} J(t, \g ; v).
	\end{equation}
	So we only need to prove the inverse inequality of \rf{21.9.1.1}, i.e.,
	\begin{equation}\label{21.9.1.2}
		\esssup_{v \in \cU} J(t, \g ; v)\les\esssup_{v \in \overline{\cU}^t} J(t, \g ; v).
	\end{equation}
	Note that $\overline{\cU}^t$ is dense in the space $\mathcal{U}$. Hence for each $v\in\cU$, there is a sequence $\{v^n\}_{n=1}^\i\in\overline{\cU}^t$ such that
	$$\lim_{n\ra\i}\dbE\int_t^T|v^n(s)-v(s)|^2ds=0.$$
	Moreover, there is a subsequence of  $\{v^n\}_{n=1}^\i$, still denoted by  $\{v^n\}_{n=1}^\i$ without loss of generality, which satisfies
	\begin{equation}\label{21.9.2.1}
		\lim_{n\ra\i}v^n(s)=v(s)\q \hb{a.s.}
	\end{equation}
	Denote by $(Y^{t,\g;v^n}(\cd),Z^{t,\g;v^n}(\cd))$ the adapted solution of BSDE \rf{BSDE} with $v$ replaced by $v^n$.
	Applying It\^o's lemma to $|Y^{t,\g;v^n}(s)-Y^{t,\g;v}(s)|^2e^{-\b s}$ on the interval $[t,T]$, together with the monotonic condition (A5), we have
	\begin{equation*}
		\begin{aligned}
			\ds&  \dbE|Y^{t, \g ; v^{n}}(t)-Y^{t, \g ; v}(t)|^{2}
			+\dbE\int_t^T\big[\b |Y^{t, \g ; v^{n}}(s)-Y^{t, \g ; v}(s)|^{2}
			+|Z^{t, \g ; v^{n}}(s)-Z^{t, \g ; v}(s)|^{2}\big]e^{\b (s-t)}ds\\
			\ns=&\ds e^{\b T}\dbE|\Phi(X^{t, \g ; v^{n}}_T)-\Phi(X^{t, \g ; v}_T)|^{2}
			+\dbE\int_{t}^{T}2\big[Y^{t, \g ; v^{n}}(s)-Y^{t, \g ; v}(s)\big]\\
			\ns\ds & \cd\big[g(s, X^{t, \g ; v^n}_s, Y^{t,\g;v^n}(s), Z^{t,\g;v^n}(s),v^{n}(s))
			-g(s, X^{t, \g ; v}_s, Y^{t, \g ; v}(s), Z^{t, \g ; v}(s), v(s))\big]e^{\b (s-t)}ds\\
			\ns\les&\ds C_p\dbE\|X^{t, \g ; v^{n}}_T-X^{t, \g ; v}_T\|_{\dbC}^{2}
			+\dbE\int_{t}^{T}\big[\frac{6\tilde L^2}{\b}\|X^{t, \g ; v^{n}}_s-X^{t, \g ; v}_s\|_{\dbC}^{2}
			+(\frac{\b }{2}+\mu)|Y^{t, \g ; v^{n}}(s)-Y^{t, \g ; v}(s)|^2\\
			\ns\ds &\q+\frac{6\tilde L^2}{\b}|Z^{t, \g ; v^{n}}(s)-Z^{t, \g ; v}(s)|^2
			+\frac{6}{\b}|g(s,0,0,0,v^n(s))-g(s,0,0,0,v(s))|^2\big]e^{\b (s-t)}ds,
		\end{aligned}
	\end{equation*}
	where and in the rest of this paper $C_p$ is a positive constant depending only on the given parameters, whose values may be changed from line to line. Let $\b=2(1+\mu)+6\tilde L^2+1$, and then by the above inequality we have
	\begin{equation}\label{21.9.1.3}
		\begin{aligned}
			\ds  \dbE|Y^{t, \g ; v^{n}}(t)-Y^{t, \g ; v}(t)|^{2}
			\les& C_p\dbE\|X^{t, \g ; v^{n}}_T-X^{t, \g ; v}_T\|_{\dbC}^{2}
			+C_p\dbE\|X^{t, \g ; v^{n}}_s-X^{t, \g ; v}_s\|_{\dbC}^{2}\\
			\ns\ds &+C_p\dbE\int_{t}^{T}|g(s,0,0,0,v^n(s))-g(s,0,0,0,v(s))|^2\big]ds.
		\end{aligned}
	\end{equation}
	For the first and second terms on the right hand side of \rf{21.9.1.3}, by \autoref{21.9.1.4} we have
	$$\dbE\|X^{t, \g ; v^{n}}_T-X^{t, \g ; v}_T\|^{2}_{\dbC}
	+\dbE\|X^{t, \g ; v^{n}}_s-X^{t, \g ; v}_s\|^{2}_{\dbC}
	\les C_{p}\dbE\int_{t}^{T}|v^n(s)-v(s)|^{2} d s.$$
	For the third term on the right hand side of \rf{21.9.1.3}, it follows from the dominated convergence theorem and \rf{21.9.2.1} that
	$$\lim_{n\ra\i}\dbE\int_{t}^{T}\big|g(s, 0,0,0, v^{n}(s)) -g(s, 0,0,0, v(s))\big|^{2}ds=0.$$
	Hence, taking the limits on both side of \rf{21.9.1.3}, we obtain
	$$\lim_{n\ra\i}\dbE|Y^{t, \g ; v^{n}}(t)-Y^{t, \g ; v}(t)|^{2}=0.$$
	Therefore, there is a subsequence of $\{v^n\}_{n=1}^\i$, still denoted by  $\{v^n\}_{n=1}^\i$ without loss of generality, such that
	$$\lim_{n\ra\i}Y^{t, \g ; v^{n}}(t)=Y^{t, \g ; v}(t)\q \hb{a.s.}$$
	In other words,
	$$\lim_{n\ra\i}J(t,\g,v^n)=J(t,\g,v)\q \hb{a.s.}$$
	Now, by the arbitrariness of $v$ and the definition of essential supremum, we obtain \rf{21.9.1.2}, and then \rf{21.9.1.0} follows.
	
	\ms
	
	\textbf{Step 2}.  We prove
	\begin{equation}\label{21.9.2.2}
		\esssup_{v \in \overline{\cU}^t} J(t, \g ; v)=\esssup_{v \in \cU^t} J(t, \g ; v).
	\end{equation}
	Obviously,
	\begin{equation}\label{wz1}
		\esssup_{v \in \overline{\cU}^t} J(t, \g ; v)\ges\esssup_{v \in \cU^t} J(t, \g ; v).
	\end{equation}
	So we only need to prove the inverse inequality of \rf{wz1}. By Lemma 3.1 in Chen--Wu \cite{Chen-Wu-12}, we have
	$$
	J(t, \g; v)=J(t, \g ; \sum_{j=1}^{N} 1_{A_{j}} v^{j})
	=\sum_{j=1}^{N} 1_{A_{j}} J(t, \g ; v^{j}),\q \forall v\in  \overline{\cU}^t,
	$$
	where $v^j$ is $\mathscr{F}_t^s$-measurable for $j=1,2,\cdots,N$. It is well known that
	$Y^{t,\g;v^{j}}(t)=J(t, \g ; v^{j})$, as the solution of BSDE, is deterministic in this case. Without lose of generality, assume
	$$J(t,\g;v^1)\ges J(t,\g;v^j),\q j=2,3,...,N.$$
	Then
	$$
	J(t, \g ; v) \les J(t, \g ; v^{1}) \les \esssup_{v \in \cU^t} J(t, \g ; v).
	$$
	Again, from the arbitrariness of $v$, we obtain
	$$\esssup_{v \in \overline{\cU}^t} J(t, \g; v)
	\les \esssup_{v \in \cU^{t}} J(t, \g; v).$$
	Hence the desired result \rf{21.9.2.2} holds.
	
	\ms
	
	Finally, \rf{21.9.1.0} and \rf{21.9.2.2} lead to
	\begin{equation}\label{wz2}
		\esssup_{v \in \cU} J(t, \g; v)
		= \esssup_{v \in \cU^{t}} J(t, \g; v),
	\end{equation}
	which implies that the value function $u$ defined in \rf{value-function} is a deterministic function.
\end{proof}

\begin{proof}[Proof of \autoref{9.23.2}]
	By \autoref{9.20.2}, for any $t\in[0,T]$, $\g,\bar\g\in\dbC$ and $v\in\cU$, we have
	\begin{align}\label{9.20.3}
		|J(t,\g;v)|=|Y^{t,\g;v}(t)|\les C\big( 1+\|\g\|_{\dbC}\big)
	\end{align}
	and
	\begin{align}\label{9.20.4}
		|J(t,\g;v)-J(t,\bar\g;v)|\les C\|\g-\bar\g\|_{\dbC}.
	\end{align}
	On the other hand, for arbitrary $\e>0$, there exist $v,\bar v\in\cU$ such that
	\begin{equation}\label{9.22.1}
		J(t,\g; v)\les u(t,\g)\les J(t,\g;\bar v)+\e\ \ \ {\rm and}\ \ \
		J(t,\bar\g;\bar v)\les u(t,\bar\g)\les J(t,\bar\g; v)+\e,
	\end{equation}
	which combining \rf{9.20.3} implies
	\begin{align*}
		-C\big( 1+\|\g\|_{\dbC}\big)
		\les J(t,\g; v)\les u(t,\g)\les J(t,\g;\bar v)+\e
		\les C\big( 1+\|\g\|_{\dbC}\big)+\e.
	\end{align*}
	By the arbitrariness of $\e$, \rf{9.21.2} holds.
	For \rf{9.21.1}, using \rf{9.22.1} again, we have
	\begin{align*}
		J(t,\g;v)-J(t,\bar\g,v)-\e\les u(t,\g)-u(t,\bar\g)\les J(t,\g;\bar v)-J(t,\bar\g; \bar v)+\e,
	\end{align*}
	which combining \rf{9.20.4} implies that
	\begin{align*}
		\ds |u(t,\g)-u(t,\bar\g)|\les\max\{|J(t,\g;v)-J(t,\bar\g,v)|,\ |J(t,\g;\bar v)-J(t,\bar\g; \bar v)|\}+\e\les C\|\g-\bar\g\|_{\dbC}+\e.
	\end{align*}
	Then \rf{9.21.1} follows from the arbitrariness of $\e$.
\end{proof}

\begin{proof}[Proof of \autoref{9.23.1}]
	The proof is divided into three steps.
	
	\ms
	
	\textbf{Step 1}. We consider a simple $\dbC$-valued random variable $\Upsilon\in L^{2}_{\dbF}(\Om;\dbC_t)$ with a form below
	\begin{equation}\label{12.18.1}
		\Upsilon=\sum^N_{i=1}\g^iI_{A_i},
	\end{equation}
	where
	$\g^i\in\dbC$ with $1\les i\les N$, $N\in\mathbb{N}$ and $\{A_i\}^N_{i=1}$ is a partition of $(\Om,\sF_t)$.
	For each $i=1,2,...,N,$ denote by $(Y^{t, \g^i ; v}(\cd),Z^{t, \g^i ; v}(\cd))$ the adapted solution of BSDE
	\begin{equation}\label{12.18.2}
		\begin{aligned}
			\ds Y^{t,\g^i;v}(s)=&\ \Phi(X_T^{t,\g^i;v})
			+\int_{s}^{T} g(r,X_r^{t,\g^i;v},Y^{t,\g^i;v}(r),Z^{t,\g^i;v}(r),v(r))dr\\
			\ns\ds& -\int_{s}^{T} Z^{t,\g^i;v}(r)dW(r),\q~ s\in[t,T].
		\end{aligned}
	\end{equation}
	Multiplying by $I_{A_i}$  on both side of BSDE \rf{12.18.2} and merging the corresponding terms, we get
	\begin{equation*}
		\begin{aligned}
			\ds &\ \sum_{i=1}^{N} 1_{A_{i}}Y^{t,\g^i;v}(s)\\
			\ns=&\ds \Phi\left(\sum_{i=1}^{N} 1_{A_{i}}X_T^{t,\g^i;v}\right)+\int_{s}^{T} g\left(r,\sum_{i=1}^{N} 1_{A_{i}}X_r^{t,\g^i;v},\sum_{i=1}^{N} 1_{A_{i}}Y^{t,\g^i;v}(r),\sum_{i=1}^{N} 1_{A_{i}}Z^{t,\g^i;v}(r),v(r)\right)dr\\
			\ns\ds&\ \ \ \ \ -\int_{s}^{T} \sum_{i=1}^{N} 1_{A_{i}}Z^{t,\g^i;v}(r)dW(r), \q~ s\in[t,T].
		\end{aligned}
	\end{equation*}
	By the existence and uniqueness of BSDE, it yields that
	\begin{equation*}
		Y^{t, \Upsilon ; v}(s)=\sum_{i=1}^{N} 1_{A_{i}}Y^{t,\g^i;v}(s)\ \ \ {\rm and}\ \ \
		Z^{t, \Upsilon ; v}(s)=\sum_{i=1}^{N} 1_{A_{i}}Z^{t,\g^i;v}(s),\q~s\in[t,T].
	\end{equation*}
	Then, by the definition of $J(\cd)$, we obtain
	$$
	Y^{t, \Upsilon ; v}(t)=\sum_{i=1}^{N} 1_{A_{i}} Y^{t, \g^{i} ; v}(t)
	=\sum_{i=1}^{N} 1_{A_{i}} J(t, \g^{i} ; v)
	=J(t, \sum_{i=1}^{N} 1_{A_{i}} \g^{i} ; v)=J(t, \Upsilon ; v).
	$$
	So the conclusion holds for the simple case.
	
	\ms
	
	\textbf{Step 2}. For any bounded $\Phi \in L^{2}_{\dbF}(\Om;\dbC_t)$, by \autoref{9.22.2}
	there is a sequence of simple $\dbC$-valued random variable
	\begin{equation}\label{wz3}
		\Upsilon^m=\sum_{i=1}^{m} 1_{A_{i}} \g^{i},
	\end{equation}
	which converges to $\Phi$ in $L^{2}_{\dbF}(\Om;\dbC_t)$. Here
	$\g^i\in\dbC$ with $1\les i\les m$, $m\in\mathbb{N}$ and $\{A_i\}^m_{i=1}$ is a partition of $(\Om,\sF_t)$. By \autoref{9.20.2}, we have
	$$
	\mathbb{E}|Y^{t, \Phi; v}(t)-Y^{s, \Upsilon^{m} ; v}(t)|^2
	\les C \mathbb{E}\|\Phi-\Upsilon^{m}\|^2_{\dbC} \rightarrow 0\ \ \text { as }\ m \rightarrow \infty.
	$$
	On the other hand, for the cost functional $J$ defined in \rf{cost-functional}, by \rf{9.20.4} we have
	$$
	\mathbb{E}|J(t, \Phi; v)-J(t, \Upsilon^m ; v)|
	\les C \mathbb{E}\|\Phi-\Upsilon^{m}\|^2_{\dbC} \rightarrow 0\ \ \text { as }\ m \rightarrow \infty,
	$$
	which combining the result $Y^{t, \Upsilon^{m} ; v}(t)=J(t, \Upsilon^{m} ; v)$ implies that the conclusion holds in this case.
	
	\ms
	
	\textbf{Step 3}. For any $\Psi\in L^{2}_{\dbF}(\Om;\dbC_t)$, we define $\Phi^{m}=(\Psi_t(r) \wedge m) \vee(-m)$ for $m\in\mathbb{N}$, $r\in[t-\delta,t]$. Then $\Phi^{m}\in L^{2}_{\dbF}(\Om;\dbC_t)$ is bounded and
	$$
	\mathbb{E}|Y^{t, \Psi; v}(t)-Y^{s, \Phi^{m} ; v}(t)|^2 \les C \mathbb{E}\|\Psi-\Phi^{m}\|^2_{\dbC} \rightarrow 0\ \ \text { as }\ m \rightarrow \infty.
	$$
	The desired result follows by a similar argument as Step 2.
\end{proof}

\begin{proof}[Proof of \autoref{9.23.6}] The proof is divided into three steps.

	\textbf{Step 1}.
	For the simple $\dbC$-valued random variable $\Upsilon$ as \rf{12.18.1}, \rf{9.23.3} follows from
	$$
	Y^{t, \Upsilon ; v}(t)=\sum_{i=1}^{N} 1_{A_{i}} Y^{t, \g^{i} ; v}(t)=\sum_{i=1}^{N} 1_{A_{i}}J(t, \g^{i} ; v)
	\les \sum_{i=1}^{N} 1_{A_{i}} u(t, \g^{i})=u(t, \Upsilon), \q \forall v\in\cU.
	$$
	To prove (\ref{9.23.4}), note that for each $\g^{i}$, by (\ref{wz2}) there exists an admissible control $v_i\in\mathcal{U}^t$ such that
	$$u(t,\g^{i})\les Y_t^{t,\g^{i};v_i}+\varepsilon.$$
	Taking $v=\sum\limits_{i=1}^Nv_iI_{A_i}\in\mathcal{U}$ we have
	\begin{equation*}\label{pz41}
		Y_t^{t,\Upsilon;v}+\varepsilon=\sum_{i=1}^N(Y_t^{t,\g^i;v_i}+\varepsilon)I_{A_i}\ges\sum_{i=1}^Nu(t,\g^i)I_{A_i}=u(t,\Upsilon).
	\end{equation*}
	So (\ref{9.23.4}) also holds for the simple $\dbC$-valued random variable.
	\ms
	
	\textbf{Step 2}. For any $\Psi \in L^{2}_{\dbF}(\Om;\dbC_t)$, we combine Steps 2-3 in Lemma \ref{9.23.1} by diagonal rule to get a sequence of simple $\dbC$-valued random variables $\Upsilon^{m}$ with a form as \eqref{wz3}, $m\in\mathbb{N}$, which converges to $\Psi$ in the space $L^{2}_{\dbF}(\Om;\dbC_t)$. Hence, by \autoref{9.20.2} and \autoref{9.23.2}, we have for any $v\in\cU$,
	$$
	\lim _{m \rightarrow \infty} Y^{t, \Upsilon^{m} ; v}(t)=Y^{t, \Phi ; v}(t)\ \text { a.s. }\q \hb{ and }\q
	\lim _{m \rightarrow \infty} u(t, \Upsilon^{m})=u(t, \Phi)\ \text { a.s., }
	$$
	which combining $Y^{t, \Upsilon^m ; v}(t)\les u(t,\Upsilon^m)$, $m=1,2,\cdots$, implies that \rf{9.23.3} holds.

	To prove (\ref{9.23.4}), note that for arbitrary $\e>0$, there exists $m\in\mathbb{N}$ such that
	$$\mathbb{E}\|\Psi-\Upsilon^m\|^2_{\dbC}\les \frac{\e}{3C}.$$
	By \autoref{9.20.2} and \autoref{9.23.2} again, it yields that for any $u\in\cU$,
	\begin{equation}\label{9.23.5}
		|Y^{t,\Psi;v}(t)-Y^{t,\Upsilon^m;v}(t)|\les \frac{\e}{3}\qq\hb{and}\qq |u(t,\Psi)-u(t,\Upsilon^m)|\les\frac{\e}{3} .
	\end{equation}
	Then for each $\g^i$, by (\ref{wz2}) there is an admissible control $v^i\in\cU^t$ such that
	$$u(t,\g^i)\les Y^{t,\g^i;v^i}(t)+\frac{\e}{3}.$$
	Now, we set $v=\sum\limits_{i=1}^{m} 1_{A_{i}} v^i$, and by \rf{9.23.5} it yields that
	\begin{equation*}
		\begin{aligned}
			\ds Y^{t,\Psi; v}(t) & \ges-|Y^{t, \Upsilon^m ; v}(t)-Y^{t, \Psi, v}(t)|+Y^{t, \Upsilon^m ; v}(t) \\
			\ns\ds & \ges-\frac{\e}{3}+\sum_{i=1}^{m} 1_{A_{i}} Y^{t, \g^{i} ; v^{i}}(t) \\
			\ns\ds & \ges-\frac{\e}{3}+\sum_{i=1}^{m} 1_{A_{i}}\big[u(t, \g^{i})-\frac{\e}{3}\big]\\
			\ns\ds & =-\frac{2}{3} \e+\sum_{i=1}^{m} 1_{A_{i}} u(t, \g^{i}) \\
			\ns\ds &=-\frac{2}{3} \e+u(t, \Upsilon^m) \ges -\e+u(t, \Psi),
		\end{aligned}
	\end{equation*}
	which implies that \rf{9.23.4} holds.
	%
\end{proof}

Next, we give the detailed proofs to the lemmas in Subsection \ref{Sec4}.

\begin{proof}[Proof of \autoref{gn-g}]
	Note that $\int_{\mathbb{R}} \rho_{n}(a) d a=1$. Then we have
	$$
	g_{n}(t, \g, y, v)-g(t, \g, y, v)=\int_{\mathbb{R}}(g(t, \g, y-a, v)-g(t, \g, y, v)) \rho_{n}(a) d a.
	$$
	On the other hand, for each compact set $K \subset[0, T] \times \dbC \times \mathbb{R} \times U$, there is a compact set $\widetilde{K}$ such that $(t, \g, y-a, v) \in \widetilde{K}$ for any $(t,\g, y, v) \in K$ and $a \in[-1,1] .$
	Notice that since $g(t, \g, y, v)$ is continuous with respect to $(t, y, v)$ and Lipschitz continuous with respect to $\g$, $g(t, \g, y, v)$ is continuous and uniformly continuous with respect to $(t, \g, y, v)$ in the compact set $\widetilde{K}$. Therefore, for arbitrary $\e>0$,  as $n$ is sufficiently large, we have
	\begin{align*}
\sup _{(t, \g, y, v) \in K}|g_{n}(t, \g, y, v)-g(t, \g, y, v)|&\les\sup _{(t, \g, y, v) \in K} \int_{|a| \les \frac{1}{n}}|g(t, \g, y-a, v)-g(t, \g, y, v)| \rho_{n}(a) d a \\
		&\les\varepsilon \int_{|a| \les \frac{1}{n}} \rho_{n}(a) d a=\varepsilon.
	\end{align*}
	So the desired conclusion is derived.
\end{proof}

\begin{proof}[Proof of \autoref{Yn-Y}]
	Firstly, it is easy to check that the smootherized generator $g_n$ satisfies the assumptions (A3)-(A6).
	Then, using It\^o's lemma to $|Y^{t,\g,n;v}(s)-Y^{t,\g;v}(s)|^2e^{\b(s-t)}$ on the interval $[t,T]$, we have for any $(t,\g)\in K$,
	\begin{equation*}
		\begin{aligned}
			\ds &\dbE[|Y^{t, \g, n ; v}(t)-Y^{t, \g ; v}(t)|^{2}]
			+\dbE\int_{t}^{T}\left[\beta\left|Y^{t, \g, n ; v}(s)-Y^{t, \g ; v}(s)\right|^{2}
			+\left|Z^{t, \g, n ; v}(s)-Z^{t, \g ; v}(s)\right|^{2}\right] e^{\beta(s-t)} d s \\
			\ns=&\ds 2\dbE\int_{t}^{T}  \big[Y^{t, \g, n ; v}(s)-Y^{t, \g ; v}(s)\big]
			\cd \big[g_{n}(s, X_{s}^{t, \g ; v}, Y^{t, \g,n ; v}(s), v(s))
			-g(s, X_{s}^{t, \g ; v}, Y^{t, \g ; v}(s), v(s))\big] e^{\beta(s-t)} d s\\
			\ns=&\ds 2\dbE\int_{t}^{T}  \big[Y^{t, \g, n ; v}(s)-Y^{t, \g ; v}(s)\big]
			\cd \big[g_{n}(s, X_{s}^{t, \g ; v}, Y^{t, \g,n ; v}(s), v(s))
			-g_n(s, X_{s}^{t, \g ; v}, Y^{t, \g ; v}(s), v(s))\big] e^{\beta(s-t)} d s\\
			\ns&\ds+\dbE\int_{t}^{T} 2\big[Y^{t, \g, n ; v}(s)-Y^{t, \g ; v}(s)\big]
			\cd \big[g_{n}(s, X_{s}^{t, \g ; v}, Y^{t, \g ; v}(s), v(s))
			-g(s, X_{s}^{t, \g ; v}, Y^{t, \g ; v}(s), v(s))\big] e^{\beta(s-t)} d s\\
			\ns\les&\ds 2\dbE\int_{t}^{T} \mu |Y^{t, \g, n ; v}(s)-Y^{t, \g ; v}(s)|^2 e^{\beta(s-t)} ds
			+\dbE\int_t^T \Big[\frac{\b}{2}|Y^{t, \g, n ; v}(s)-Y^{t, \g ; v}(s)|^2\\
			\ns\ds &\q +\frac{2}{\b}\big| g_{n}(s, X_{s}^{t, \g ; v}, Y^{t, \g ; v}(s), v(s))
			-g(s, X_{s}^{t, \g ; v}, Y^{t, \g ; v}(s), v(s))\big|^2\Big]e^{\beta(s-t)} ds,
		\end{aligned}
	\end{equation*}
	where in the last inequality we have used H\"{o}lder and inequality (A5) satisfied by $g_n$.
	Now, taking $\b=4\mu+2$ we have
	\begin{equation*}
		\begin{aligned}
			\ds & \dbE[|Y^{t, \g, n ; v}(t)-Y^{t, \g ; v}(t)|^{2}] \\
			\ns\les&\ds  C_{p} \dbE\int_{t}^{T}\big|g_{n}(s, X_{s}^{t, \g ; v}, Y^{t, \g; v}(s), v(s) )
			-g(s, X_{s}^{t, \g ; v}, Y^{t, \g ; v}(s), v(s))\big|^{2} d s \\
			\ns=&\ds C_{p} \dbE\int_{t}^{T}|g_{n}-g|^{2}
			1_{\{\|X_{s}^{t, \g ; v}\|_{\dbC}\ges N\} \cup\{\sup\limits_{s \in[t, T]}|Y^{t, \g ; v}(s)| \ges N\}\}} d s \\
			\ns&\ds+C_{p} \dbE\int_{t}^{T}|g_{n}-g|^{2} 1_{\{\{\|X_{s}^{t, \g ; v}\|_{\dbC}<N\} \cap\{\sup\limits_{s \in[t, T]}|Y^{t, \g ; v}(s)|<N\}\}} d s \\
			\ns\les &\ds C_{p} \dbE\int_{t}^{T}|g_{n}-g|^{2} 1_{\|X_{s}^{t, \g ; v}\|_{\dbC} \ges N\}} d s
			+C_{p} \dbE\int_{t}^{T}|g_{n}-g|^{2} 1_{\{\sup\limits_{s \in[t, T]}|Y^{t,\g ; v}(s)| \ges N\}} d s \\
			\ns&\ds+C_{p} \dbE\int_{t}^{T}|g_{n}-g|^{2} 1_{\{\{\|X_{s}^{t, \g ; v}\|_{\dbC}<N\} \cap\{\sup\limits_{s \in[t, T]}|Y^{t, \g ; v}(s)|<N\}\}} d s.
		\end{aligned}
	\end{equation*}
	For simplicity, we define
	\begin{align*}
		\ds  J_1\deq&\ \dbE\int_{t}^{T}
		\big|g_{n}(s, X_{s}^{t, \g ; v}, Y^{t, \g ; v}(s), v(s) )
		-g(s, X_{s}^{t, \g ; v}, Y^{t, \g ; v}(s), v(s))\big|^{2}
		1_{\|X_{s}^{t, \g ; v}\|_{\dbC} \ges N\}} d s,\\
		\ns\ds  J_2\deq&\ \dbE\int_{t}^{T}
		\big|g_{n}(s, X_{s}^{t, \g ; v}, Y^{t, \g ; v}(s), v(s) )
		-g(s, X_{s}^{t, \g ; v}, Y^{t, \g ; v}(s), v(s))\big|^{2}
		1_{\{\sup\limits_{s \in[t, T]}|Y^{t,\g ; v}(s)| \ges N\}} d s,\\
		\ns\ds  J_3\deq&\ \dbE\int_{t}^{T}
		\big|g_{n}(s, X_{s}^{t, \g ; v}, Y^{t, \g ; v}(s), v(s) )
		-g(s, X_{s}^{t, \g ; v}, Y^{t, \g ; v}(s), v(s))\big|^{2}\\
		&\qq\q \cd 1_{\{\{\|X_{s}^{t, \g ; v}\|_{\dbC}<N\} \cap\{\sup\limits_{s \in[t, T]}|Y^{t, \g ; v}(s)|<N\}\}} d s.
	\end{align*}
	Next we will deal with $J_1$, $J_2$ and $J_3$ in turn.
	
	\ms
	
	First, for $J_1$, by (A6) we have
	\begin{align*}\label{9.30.1}
		\ds  &\sup_{(t,\g)\in K}J_1 \nn\\
		\ns\les\ds  & \sup_{(t,\g)\in K} 2\dbE\int_{t}^{T}
		\Big(\big|g_{n}(s, X_{s}^{t, \g ; v}, Y^{t, \g ; v}(s), v(s) )\big|^2
		+\big|g(s, X_{s}^{t, \g ; v}, Y^{t, \g ; v}(s), v(s))\big|^{2}\Big)
		1_{\|X_{s}^{t, \g ; v}\|_{\dbC} \ges N\}} d s \nn\\
		\ns\les\ds  & \sup_{(t,\g)\in K} C_p\dbE\int_{t}^{T}
		\Big(1+\|X_{s}^{t, \g ; v}\|_{\dbC}^2
		+|Y^{t, \g; v}(s)|^{2p}\Big)
		1_{\|X_{s}^{t, \g ; v}\|_{\dbC} \ges N\}} d s \nn\\
		\ns\les\ds  & \sup_{(t,\g)\in K} C_p
		\Big[\dbE\int_{t}^{T} \Big(1+\|X_{s}^{t, \g ; v}\|_{\dbC}^4
		+|Y^{t, \g; v}(s)|^{4p}\Big)ds\Big]^{\frac{1}{2}}
		\cdot \Big[\sup_{(t,\g)\in K}\dbP(\|X_{s}^{t, \g ; v}\|_{\dbC} \ges N)\Big]^{\frac{1}{2}}.
	\end{align*}
	By Chebychev's inequality and \autoref{21.9.1.4}, we have for any $N>0$,
	$$\dbP(\|X_{s}^{t, \g ; v}\|_{\dbC} \ges N)
	\les \frac{1}{N^2}\dbE\|X_{s}^{t, \g ; v}\|^2_{\dbC}\les \frac{C_p}{N^2}(1+\|\g\|^2_{\dbC}).$$
	Note that the initial value $\g$ is bounded, so for arbitrary $\d>0$, when $N$ is sufficiently large, it yields that
	\begin{equation}\label{9.30.2}
		\sup_{(t,\g)\in K}\dbP(\|X_{s}^{t, \g ; v}\|_{\dbC} \ges N)\les \d.
	\end{equation}
	In addition, Propositions \ref{9.22.2} and \ref{21.9.1.4} imply that
	$$ \sup_{(t,\g)\in K}
	\Big[\dbE\int_{t}^{T} \Big(1+\|X_{s}^{t, \g ; v}\|_{\dbC}^4
	+|Y^{t, \g; v}(s)|^{4p}\Big)ds\Big]^{\frac{1}{2}}\les \i.$$
	This, together with \rf{9.30.2}, indicates that  for arbitrary $\e>0$, there is a sufficiently large $N_1$ such that when $N\ges N_1$, we have for all $n\in\dbN$,
	$$J_1\les\e,\q\hb{uniformly in $K$}.$$
	
	For $J_2$, \autoref{9.20.2} implies
	$$\dbE\Big[ \sup_{s\in[t,T]}|Y^{t,\g;v}(s)|^{2p} \Big]
	\les C_p(1+\|\g\|_{\dbC}^{2p}).$$
	Similar to the above discussion for $J_1$, the Chebychev's inequality leads to that for arbitrary $\d>0$,
	\begin{equation*}
		\sup_{(t,\g)\in K}\dbP(\sup_{s\in[t,T]}|Y^{t,\g;v}(s)| \ges N)\les \d.
	\end{equation*}
	Then there is a sufficiently large $N_2$ such that when $N\ges N_1$, we have for all $n\in\dbN$,
	$$J_2\les\e,\q\hb{uniformly in $K$}.$$
	
	For $J_3$, denote $\tilde{N}=N_1\vee N_2$. For $v\in\cU$, we have
	\begin{align*}
		\ds \sup_{(t,\g)\in K}J_3
		\les &\ \dbE\int_{t}^{T} \sup_{(t,\g)\in K}
		\big|g_{n}(s, X_{s}^{t, \g ; v}, Y^{t, \g; v}(s), v(s) )
		-g(s, X_{s}^{t, \g ; v}, Y^{t, \g ; v}(s), v(s))\big|^{2} \nn\\
		\ns\ds &\qq\q \cd 1_{\{\{\|X_{s}^{t, \g ; v}\|_{\dbC}<\tilde N\}
			\cap\{\sup\limits_{s \in[t, T]}|Y^{t, \g ; v}(s)|<\tilde N\}\}} d s\nn \\
		\ns\ds \les &\ \dbE\ \int_t^T\sup_{(t,\g,y)\in[0,T]\ts B^{\dbC}_{\tilde N}\ts[-\tilde N,\tilde N]}|g_n(t,\g,y,v)-g(t,\g,y,v)|^2ds,
	\end{align*}
	where $B^{\dbC}_{\tilde N}$ is the closed ball with the origin $0$ and the radium $\tilde N$ in $\dbC$. By the dominated convergence theorem and \autoref{gn-g}, we have, as $n$ is sufficiently large,
	$$J_3\les\e,\q\hb{uniformly in $K$}.$$
	
	Therefore, due to the arbitrariness of $\e$, the conclusion follows and the proof is complete.
\end{proof}

\begin{proof}[Proof of \autoref{un-u}]
	For arbitrary $\e>0$,  $t \in[0, T]$ and $\g\in\dbC$,
	there is a control $\bar v \in \mathcal{U}$ such that
	$$
	u(t, \g)<Y_{t}^{t, \g ; \bar v}+\varepsilon.
	$$
	Hence we have
	\begin{equation}\label{10.15.1}
		u(t, \g)-u_{n}(t, \g)=u(t, \g)-\sup _{v \in \mathcal{U}} Y^{t, \g, n ;  v} (t)
		\les Y^{t, \g ; \bar v}(t)+\varepsilon-Y^{t, \g, n ; \bar v}(t).
	\end{equation}
	On the other hand,  there is a control $\tilde v \in \mathcal{U}$ such that
	$$
	u_{n}(t, \g) \les Y^{t,\g, n ; \tilde v}(t)+\varepsilon,
	$$
	which gives
	\begin{equation}\label{10.15.2}
		u(t, \g)-u_{n}(t, \g) \ges u(t, \g)-Y^{t, \g, n ; \tilde v}(t)-\varepsilon
		\ges Y^{t, \g; \tilde v}(t)-Y^{t, \g, n ; \tilde v}(t)-\varepsilon.
	\end{equation}
	Combining \rf{10.15.1} and \rf{10.15.2}, we obtain
	$$
	|u(t, \g)-u_{n}(t, \g)| \les \dbE[|Y^{t, \g ; \bar v}(t)-Y^{t, \g, n ; \bar v}(t)|]
	+\dbE[|Y^{t, \g ; \tilde v}(t)-Y^{t, \g, n ; \tilde v}(t)|]+2\varepsilon.
	$$
	Note that  $\bar v$ and $\tilde v$ are given admissible controls. Then by \autoref{Yn-Y} it yields that for each compact set $K \subset[0, T] \times \dbC$,
	$$
	\lim _{n \rightarrow \infty} \sup _{(t, \g) \in K} \dbE[|Y^{t, \g ; \bar v}(t)-Y^{t, \g, n ; \bar v}(t)|]+\dbE[|Y^{t, \g ; \tilde v}(t)-Y^{t, \g, n ; \tilde v}(t)|]=0.
	$$
	Finally, owing to the arbitrariness of $\varepsilon$, we obtain the uniform convergence of the value functions $u_{n}$ to $u$ in $K$.
\end{proof}

\begin{proof}[Proof of \autoref{hn-h}]
	For any $t \in[0, T],\ \g \in\dbC,\ r \in \mathbb{R},\ p \in \mathbb{R}^{n},\ A \in \mathbb{S}^{n},\ v \in U$, we set
	$$
	\mathcal{A}=\langle p, b(t,\g,v)\rangle
	+\frac{1}{2}Tr\left(\sigma(t,\g,v)\sigma^*(t,\g,v)A\right)+g_n(t,\g,r,v)
	$$
	and
	$$
	\mathcal{B}=\langle p, b(t,\g,v)\rangle
	+\frac{1}{2}Tr\left(\sigma(t,\g,v)\sigma^*(t,\g,v)A\right)+g(t,\g,r,v).
	$$
	Note that
	$$
	H_{n}-H=\sup _{v \in U} \mathcal{A}-\sup _{v \in U} \mathcal{B} \les \sup _{v \in U}(\mathcal{A}-\mathcal{B})=\sup _{v \in U}\left(g_{n}-g\right) \les \sup _{v \in U}\left|g_{n}-g\right|
	$$
	and
	$$
	H-H_{n}=\sup _{v \in U} \mathcal{B}-\sup _{v \in U} \mathcal{A} \les \sup _{v \in U}(\mathcal{B}-\mathcal{A})=\sup _{v \in U}\left(g-g_{n}\right) \les \sup _{v \in U}\left|g-g_{n}\right|,
	$$
	which implies
	$$
	\left|H_{n}-H\right| \les \sup _{v \in U}\left|g_{n}(t, \g, r, v)-g(t, \g, r, v)\right|.
	$$
	Note that for each compact set $K \subset[0, T] \times \dbC \times \mathbb{R} \times \mathbb{R}^{n} \times \mathbb{S}^{n}$, and any $(t, \g, r, p, A) \in K$, $v\in U$, there is a compact set $\widetilde K\in [0, T] \times \dbC \times \mathbb{R} \times U$ such that $(t, \g, r, v) \in\widetilde K$. So by \autoref{gn-g} we have
	$$
	\begin{aligned}
		\ds & \lim _{n \rightarrow \infty} \sup _{(t,\g, r, p, A) \in K}\left|H_{n}(t,\g, r, p, A)-H(t, \g, r, p, A)\right| \\
		\ns\ds \les & \lim _{n \rightarrow \infty} \sup _{(t, \g, r, p, A) \in K} \sup _{v \in U}\left|g_{n}(t, \g, r, v)-g(t,\g, r, v)\right| \\
		\ns\ds \les & \lim _{n \rightarrow \infty} \sup _{(t, \g, y, v) \in \widetilde{K}}\left|g_{n}(t, \g, y, v)-g(t, \g, y, v)\right|=0.
	\end{aligned}
	$$
	Therefore, the uniform convergence of $H_{n}$ to $H$ in $K$ follows from above.
\end{proof}

\begin{proof}[Proof of \autoref{10.24.2}]
	The proof is similar to the proof of Lemma 6.2 in Fleming--Soner \cite{Fleming-Soner-06}.
	Let $\phi\in C^{1,2}_{\text{Lip}}([0,T]\ts\dbC;\dbR)$ and $(\bar t,\bar\g)\in[0,T]\ts\dbC$ be a strict minimizer of $\phi-u$ with $\phi (\bar t,\bar\g)=u(\bar t,\bar\g)$.
	On one hand, due to that $u_n$ converges to $u$ uniformly in each compact subset of their domain, there exists a sequence $(t_n,\g_n)\rightarrow (\bar t,\bar\g)$ as $n\rightarrow\i$, such that $(t_n,\g_n)$ is a local minimum of $\phi-u_n$.
	On the other hand, if we set $\tilde \phi(t,\g)\deq \phi(t,\g)+u_n(t_n,\g_n)-\phi(t_n,\phi_n)$, then $(t_n,\g_n)$ is a local minimum of $\tilde\phi-u_n$ and $\tilde\phi(t_n,\g_n)=u_n(t_n,\g_n)$. Hence, from the viscosity property of $u_n$, we have that
	$$
	\frac{\partial}{\partial t} \tilde\phi(t_n, \g_n)+H_{n}\big(t_n, \g_n, \tilde\phi(t_n, \g_n), D_{x} \tilde\phi(t_n, \g_n), D_{x}^{2} \tilde\phi(t_n, \g_n)\big)\ges0.
	$$
	Finally, leting $n\rightarrow\i$ and using the uniform convergence of $H^n$ to $H$, we conclude that $u$ is a viscosity subsolution of the limiting equation \rf{11.4.1}. Similarly the supersolution property can be proved similarly.
\end{proof}

\section{Conclusion}

In conclusion, we show that under the assumptions (A1)-(A6), the dynamic programming principle for delayed stochastic recursive optimal control problem holds under the non-Lipschitz circumstances in \autoref{GDDP}.
Moreover, as proved in \autoref{u-HJB}, the value function of the delayed stochastic recursive optimal control problem is the viscosity solution of the corresponding HJB equation \rf{HJB-equation}, in the setting of non-Lipschitz generator  $g(t,\g, y, v)$ of BSDE serving as the recursive cost functional. However, when the generator $g$ also depends on $z$, the second one of the solution pair, the result is still open, and we hope to solve it in the near future.
Finally, in Section 5, to demonstrate the application of our theoretical study to mathematical finance models, a consumption-investment problem under the delayed continuous-time Epstein-Zin utility with a non-Lipschitz is presented and the HJB equation which the value function satisfies is given.


\end{document}